\documentclass[11pt,a4paper]{article}

\usepackage{fullpage}

\usepackage[numbers, sort]{natbib}

\usepackage{amsmath,amssymb,amsthm}
\usepackage{tikz}
\usepackage{siunitx}

\usepackage[colorlinks=true]{hyperref}
\hypersetup{urlcolor=blue, citecolor=blue, linkcolor=blue}

\newcommand*{\RR}{{\mathbb{R}}}
\newcommand*{\CC}{{\mathbb{C}}}
\newcommand*{\NN}{{\mathbb{N}}}

\renewcommand*{\SS}{{\mathbb{S}}}

\newcommand{\vA}{{\boldsymbol{A}}}
\newcommand{\vE}{{\boldsymbol{E}}}

\newcommand{\vH}{{\boldsymbol{H}}}
\newcommand{\vU}{{\boldsymbol{U}}}
\newcommand{\vV}{{\boldsymbol{V}}}
\newcommand{\vM}{{\boldsymbol{M}}}
\newcommand{\vN}{{\boldsymbol{N}}}
\newcommand{\vS}{{\boldsymbol{S}}}
\newcommand{\vX}{{\boldsymbol{X}}}

\newcommand{\vW}{{\boldsymbol{W}}}

\newcommand{\vXj}{{\boldsymbol{X}}_{\mkern-2mu j}}
\newcommand{\vYj}{{\boldsymbol{Y}}_{\mkern-2mu j}}
\newcommand{\vSj}{{\boldsymbol{S}}_{\mkern-2mu j}}
\newcommand{\vHj}{{\boldsymbol{H}}_{\mkern-2mu j}}

\newcommand{\vXjj}{{\boldsymbol{X}}_{\mkern-2mu j+1}}

\newcommand{\vHjj}{{\boldsymbol{H}}_{\mkern-2mu j+1}}

\newcommand{\vb}{{\boldsymbol{b}}}
\newcommand{\vc}{{\boldsymbol{c}}}
\newcommand{\vh}{{\boldsymbol{h}}}
\newcommand{\vn}{{\boldsymbol{n}}}
\newcommand{\vp}{{\boldsymbol{p}}}
\newcommand{\vt}{{\boldsymbol{t}}}
\newcommand{\vu}{{\boldsymbol{u}}}

\newcommand{\vx}{{\boldsymbol{x}}}
\newcommand{\vy}{{\boldsymbol{y}}}
\newcommand{\vz}{{\boldsymbol{z}}}
\newcommand{\vxhat}{\widehat{\boldsymbol{x}}}
\newcommand{\vyhat}{\widehat{\boldsymbol{y}}}
\newcommand{\vd}{{\boldsymbol{d}}}
\newcommand{\vP}{{\boldsymbol{P}}}

\newcommand{\veta}{{\boldsymbol{\eta}}}
\newcommand{\vlambda}{{\boldsymbol{\lambda}}}
\newcommand{\vnu}{{\boldsymbol{\nu}}}
\newcommand{\vphi}{{\boldsymbol{\varphi}}}
\newcommand{\vpsi}{{\boldsymbol{\psi}}}
\newcommand{\vzeta}{{\boldsymbol{\zeta}}}

\newcommand{\calF}{{\mathcal{F}}}

\newcommand{\I}{\ensuremath{\mathrm{i}}}
\newcommand{\E}{\ensuremath{\mathrm{e}}}
\newcommand{\D}{\ensuremath{\mathrm{d}}}
\newcommand{\hs}{\ensuremath{\mathrm{HS}}}

\newcommand{\red}[1]{{#1}}

\DeclareMathOperator{\real}{Re}

\DeclareMathOperator{\supp}{supp}
\DeclareMathOperator{\Span}{span}

\DeclareMathOperator{\Dcurl}{\mathbf{curl}}

\DeclareMathOperator{\Ddiv}{div}
\DeclareMathOperator{\DCurl}{Curl} 
\DeclareMathOperator{\DGrad}{\mathbf{Grad}}
\DeclareMathOperator{\DDiv}{Div}

\theoremstyle{plain}
\newtheorem{theorem}{Theorem}
\newtheorem{lemma}[theorem]{Lemma}

\theoremstyle{definition}
\newtheorem{definition}[theorem]{Definition}

\theoremstyle{remark}
\newtheorem{remark}[theorem]{Remark}

\begin{document}

\title{Optimal Design of Tubular Perfectly Conducting Objects in Electromagnetic Chirality}

\author{
Tilo Arens\footnote{Karlsruhe Institute of Technology (KIT), Kaiserstr.\@ 12, 76131 Karlsruhe, Germany
    {\tt tilo.arens@kit.edu, \tt roland.griesmaier@kit.edu, \tt raphael.schurr@kit.edu}},
Roland Griesmaier\footnotemark[1],
Marvin
  Kn\"oller\footnote{Department of Mathematics and Statistics, University of Helsinki, Pietari Kalmin katu 5, FI-00014, Finland
    {\tt marvin.knoller@helsinki.fi}},
Raphael Schurr\footnotemark[1]
    }

\maketitle
\begin{abstract}
   This work is about the shape optimization of long tubular objects in electromagnetic chirality (em-chirality). Em-chirality is a property of individual scattering objects or metamaterials describing their qualitatively different response to electromagnetic waves of opposite polarization handedness. The optimization is performed by a Newton-type iterative maximization of a regularized em-chirality measure with respect to the scatterer's shape. In this context, the differentiability of the object-to-far field operator map is analyzed rigorously, thereby extending previously known results on the domain derivative to the far field operator setting. Our optimal design algorithm is based on the electric field integral equation, which is employed both for the evaluation of scattered fields and for the computation of the domain derivative. Our implementation is done via the boundary element method. The numerical examples presented in this work yield strongly em-chiral scattering objects capable of exciting higher-order modes beyond the dipole regime with nonintuitive shapes that expand the known set of highly em-chiral scattering objects.
\end{abstract}

{
\small\noindent
  Mathematics subject classifications (MSC2020): 78M50, 49Q10, 78A45
  \\\noindent 
  Keywords: shape optimization, Maxwell’s equations, electromagnetic chirality, tubular objects, domain derivatives, electromagnetism, perfect conductor, boundary integral equations
  \\\noindent
Short title: Optimal Design of Perfect Conductors in Electromagnetic Chirality
}

\section{Introduction}
\label{sec:Introduction}

Electromagnetic chirality (em-chirality) was recently introduced as a framework that characterizes and quantifies the asymmetric response of scattering objects to electromagnetic fields of opposing helicities. Helicity is a physical concept that extends the notion of circular polarization handedness -- originally defined for plane waves -- to general electromagnetic fields. While for plane waves helicity manifests itself as pure left or right circular polarization, for arbitrary electromagnetic fields it is defined via the eigenvalues of the helicity operator \cite{Fernan2022}. Within the framework of em-chirality, a scattering object is classified as em-chiral when its scattering response to incident fields of one helicity differs qualitatively from its response to fields of the opposite helicity. In the most extreme scenario, a maximally em-chiral object would not interact with fields of one helicity at all, i.e., it would be invisible with respect to fields of that helicity. \red{However, we will show that -- at least within the class of perfectly conducting scatterers -- no such maximally em-chiral objects do exist. Despite this negative theoretical result regarding the extreme case, highly em-chiral} objects \red{close to the maximum} could be used \red{in practice} as building blocks for metamaterials to manipulate the propagation of light and to trigger astonishing phenomena that would be impossible to obtain with more generic objects. Application areas cover novel chiral metamaterials \cite{FerRoc2019, GanThi2009, GanJus2012}, the construction of (broadband) circular polarizers \cite{HenSch2017, KadMil2019}, synthetic 3D vision systems \cite{UreChe2011} or chiral bio-sensing \cite{SchDre2012, BenMao2013}.

The archetype of an em-chiral object is a helix. Such objects (or variants of it) have been studied intensively in the context of em-chirality (see, e.g., \cite{FerFruRoc2016, GarHam2022, SchDre2012}). While optimized helices are highly effective for various scenarios it is not clear whether more general tubular objects can outperform them. It remains largely unexplored whether the design space of highly em-chiral objects contains other, equally or even more performant geometries. In this work we present a fully three-dimensional shape optimization scheme for tubular electromagnetic scatterers with a free form spine curve and apply it to unveil new designs of perfectly conducting highly em-chiral objects. The overall idea is to optimize a measure of em-chirality introduced in \cite{FerFruRoc2016} with respect to the individual scatterer's shape. Our novel objects deviate strongly from a helical shape and provide new perspectives on how highly em-chiral objects may look like. Although our shape optimizations are restricted to perfectly conducting objects, we emphasize that all the results and algorithms presented in this paper are also applicable to penetrable obstacles with transmission boundary conditions \cite{Schurr2026}.

The treatment of helices and their potential as building blocks in metamaterials can be traced back to Lindman's studies in the early 20th century (see \cite{LinSih1992} for a historical perspective). In the antenna community, helical wires radiating circularly polarized fields were first reported in \cite{Kraus1949, Wheel1947}. Ever since, metamaterials made out of helix building blocks have become a versatile platform for controlling the polarization of electromagnetic waves across wide ranges of the electromagnetic spectrum (see, e.g., \cite{GanThi2009, HenSch2017, KasWeg2016, Pendry2004, WanChe2016, WuLi2011} and the references therein).

Earlier works on optimizing individual scattering objects regarding the chirality measure from \cite{FerFruRoc2016} rely on the assumption that the cross section of the tubular object is extremely small with respect to the wavelength \cite{AreGri2021, FerGri2023, Knoel2023}. In this particular case, an asymptotic perturbation formula \cite{CapGri2021} can be used to replace the necessity of numerically solving Maxwell's equations. While the reported objects are highly em-chiral for one particular frequency, their extremely thin nature complicates a possible fabrication process immensely. In our recent work \cite{AreKnoSch2026}, we have introduced domain derivatives for thick tubular objects and applied them in an inverse scattering problem. These shape derivatives are now used in the present work for establishing Fr\'echet differentiability of the shape-to-far field operator map and to optimize thick tubes regarding their em-chirality. This approach comes at a significantly higher cost than that of \cite{AreGri2021, FerGri2023} as now many full Maxwell problems need to be solved numerically to represent the Fr\'echet derivative, but it allows to treat substantially thicker tubular objects.

This work is structured as follows. In the second section we describe electromagnetic scattering in the presence of a perfectly conducting object. Furthermore, we recall the fundamental properties of em-chirality and present the em-chirality measure that is used for the shape optimization. The third section introduces the Fr\'echet derivative of the far field operator with respect to perturbations of the scattering object. In Section~\ref{sec:OptDesignAlgorithm} we formulate the optimal design algorithm. We discuss the Newton-type algorithm, introduce the penalty terms that regularize the optimization and compute corresponding Fr\'echet derivatives.
In Section~\ref{sec:NumericalExamples} we perform the shape optimizations for optimizing perfect conductors regarding their em-chirality.
We describe our implementation that is based on boundary element methods, show convergence histories of our shape optimization and analyze the final designs of the optimization through frequency scans.
\red{In the appendix we establish a new uniqueness result for electromagnetic inverse obstacle scattering that is used to rule out the existence of maximally em-chiral objects in Section~\ref{sec:ScatProb}.}

\section{The Scattering Problem for an EM-Chiral Scatterer}
\label{sec:ScatProb}

The fundamental underlying problem considered in this paper is the scattering of a time-harmonic electromagnetic wave by a perfectly conducting obstacle given as a bounded domain $D \subseteq \RR^3$ such that~$\RR^3 \setminus \overline{D}$ is connected. We will assume throughout the paper that the boundary of the obstacle is $C^1$ smooth. This will be a necessary condition when considering domain derivatives in Section \ref{sec:Derivative}, not withstanding that many of the results mentioned in the present section are in fact true for Lipschitz domains.

Denote by $\omega > 0$ the angular frequency, and by $\varepsilon > 0$ and $\mu > 0$ the electric permittivity and magnetic permeability, respectively, which will be assumed constant throughout the exterior medium. The total field $(\vE$, $\vH)$ is a solution to the time-harmonic Maxwell system
\begin{equation}
    \label{eq:maxwell}
    \Dcurl \vE - \I \omega \mu \, \vH = 0, \qquad
    \Dcurl \vH + \I \omega \varepsilon \, \vE = 0 \qquad
    \text{in } \RR^3 \setminus \overline{D} \, ,
\end{equation}
and satisfies the perfectly conducting boundary condition
\begin{equation}
 \label{eq:bc_perf_cond}
 \vE \times \vnu = 0 \qquad \text{on } \partial D \, .
\end{equation}
Here, $\vnu$ denotes the outward drawn normal of length $1$ to $\partial D$.

In order for the scattering problem to be well-posed, we additionally require a radiation condition. Denoting by $(\vE^i, \vH^i)$ the incident field, a solution to \eqref{eq:maxwell} throughout $\RR^3$, the scattered field $(\vE^s, \vH^s) = (\vE, \vH) - (\vE^i, \vH^i)$ is assumed to satisfy the Silver--M{\"u}ller radiation condition,
\begin{equation}
    \label{eq:SMRC}
    \lim_{\vert \vx \vert \rightarrow \infty} \vert \vx \vert \left( \sqrt{\mu} \, \vH^s(\vx) \times \vxhat - \sqrt{\varepsilon} \, \vE^s(\vx) \right) = 0
\end{equation}
uniformly with respect to $\vxhat = {\vx}/{\vert \vx \vert} \in \SS^2 = \{ \vx \in \RR \;\colon | \vx | = 1 \}$.

Throughout, we will consider weak solutions to \eqref{eq:maxwell}--\eqref{eq:SMRC}. Let $B_R \subseteq \RR^3$ denote the open ball of radius $R$ centered at the origin such that $\overline{D} \subseteq B_R$ and set $\Omega = B_R \setminus \overline{D}$. We will make frequent use of the Sobolev spaces $ H(\Dcurl, \Omega)$, $H(\Dcurl^2, \Omega)$, and for some connected boundary section $\Gamma \subseteq \partial \Omega$, the boundary spaces $L^2_t(\Gamma)$, $H^{-1/2}(\DDiv, \Gamma)$, $H^{-1/2}(\DCurl, \Gamma)$ as well as the tangential trace operators $\gamma_t$, $\gamma_T$. We refer to \cite{BufCosShe2002, BufHip2003, KirHet2015, Monk2003} for the relevant definitions. We also use the duality pairing
\begin{equation*}
	\langle \vU , \vV \rangle_{t, \Gamma} = \int_{\Gamma} \vU \cdot \overline{(\vnu \times \vV)} \; \D s \, , \qquad \vU , \, \vV \in L^2_t(\Gamma) \, ,
\end{equation*}
and note that $H^{-1/2}(\DDiv, \Gamma)$ is self-dual with respect to the continuous extension of this sesquilinear form. Finally, we denote by $\Lambda$ the Calderón operator for the exterior of $B_R$. See \cite[Sec. 9.4.1]{Monk2003} for the definition of this operator.

Define
\begin{equation*}
	H_{\mathrm{pc}}(\Dcurl, \Omega) = \{ \vE \in H(\Dcurl, \Omega) \;\colon \gamma_t \vE = 0 \text{ on } \partial D \} \, ,
\end{equation*}
and for $\vU$, $\vV \in H_{\mathrm{pc}}(\Dcurl, \Omega) $, the forms
\begin{align*}
 \mathcal{A}(\vU,\vV) & = \int_{\Omega} \Big( \Dcurl \vU \cdot \overline{\Dcurl \vV} - \omega^2 \varepsilon \mu \, \vU \cdot \overline{\vV} \Big) \D\vx + \I \omega \mu \left\langle \Lambda(\gamma_t \vU), \gamma_t \vV \right\rangle_{t, \partial B_R} \, , \\
 \ell(\vV) & = \I \omega \mu \left\langle \Lambda(\gamma_t \vE^i) + \gamma_N \vE^i , \gamma_t \vV \right\rangle_{t, \partial B_R} \, ,
\end{align*}
where $\gamma_N$ denotes the electric-to-magnetic trace operator $\gamma_N \vU = 1/(\I \omega \mu) \, \gamma_t \Dcurl \vU$. The weak formulation of \eqref{eq:maxwell}--\eqref{eq:SMRC} is to find $\vE \in H_{\mathrm{pc}}(\Dcurl, \Omega)$ such that
\begin{align}
  \label{eq:weakFormOps}
  \mathcal{A}(\vE,\vV) = \ell(\vV) \qquad \text{for all } \vV \in H_{\mathrm{pc}}(\Dcurl, \Omega) \, .
\end{align}
Problem \eqref{eq:weakFormOps} is uniquely solvable (see \cite[Thm. 10.7]{Monk2003}). Note that due to the boundedness of the Calderón operator and the tangential traces, there holds
\begin{align}
  \label{eq:estimate_for_ell}
    \| \ell \|_{H_{\mathrm{pc}}(\Dcurl, \Omega)^*} \leq C \| \vE^i \|_{H(\Dcurl, \Omega)}
\end{align}
for every incident field $\vE^i \in H(\Dcurl, \Omega)$.

Our main goal is to design \emph{highly electromagnetically chiral} scatterers. The exact definition of this term is based on the \emph{far field operator,} which we define next. For any direction of propagation $\vd \in \SS^2$, polarization vector $\vp \in \CC^3 \setminus \{ 0 \}$ with $\vp \cdot \vd = 0$ and wave number $k = \omega \sqrt{\varepsilon \mu}$, we consider incident plane waves
\begin{align*}
    \vE^i(\vx, \vd, \vp) = \vp \, \E^{\I k \vd \cdot \vx} \, , \qquad
    \vH^i(\vx, \vd, \vp) = \sqrt{\frac{\varepsilon}{\mu}} \, \vd \times \vp \, \E^{\I k \vd \cdot \vx} \, ,
    \qquad \vx \in \RR^3\, .
\end{align*}
We denote by $\vE^s(\,\cdot\, , \vd, \vp)$, $\vH^s(\,\cdot\, , \vd, \vp)$ the corresponding scattered fields obtained by solving \eqref{eq:weakFormOps}. Recall that as a consequence of the radiation condition, the scattered fields have the asymptotic behavior
\begin{equation}
  \label{eq:Es_asymp}
    \left.
  \begin{aligned}
    \vE^s(\vx, \vd, \vp) & = \frac{\E^{\I k \vert \vx \vert}}{4\pi \vert \vx \vert} \left( \vE^{\infty}(\vxhat, \vd, \vp) + \mathcal{O}\left(\frac{1}{\vert \vx \vert} \right) \right) \\
    \vH^s(\vx, \vd, \vp) & = \frac{\E^{\I k \vert \vx \vert}}{4\pi \vert \vx \vert} \left( \red{\vH^{\infty}}(\vxhat, \vd, \vp) + \mathcal{O}\left(\frac{1}{\vert \vx \vert} \right) \right)
    \end{aligned}
    \quad \right\}
    \qquad \vert \vx \vert \rightarrow \infty,
\end{equation} 
uniformly in all directions $\vxhat = {\vx}/{\vert \vx \vert} \in \SS^2$ with the far field \red{patterns $\vE^{\infty}(\cdot, \vd, \vp), \vH^{\infty}(\cdot, \vd, \vp) \in L^2_t(\SS^2)$} for all admissible~$\vd$,~$\vp$. \red{We note that $\vH^{\infty}(\vxhat, \vd, \vp) = \vxhat\times \vE^{\infty}(\vxhat, \vd, \vp)$.}  More generally, we use an electromagnetic Herglotz wave pair with density $\vphi \in L_t^2(\SS^2)$,
\begin{equation}
 \label{eq:def_herglotz_wave_pair}
 \begin{aligned}
    \vE^i[\vphi](\vx) & = \int_{\SS^2} \vphi(\vd) \, \E^{\I k \vd \cdot \vx} \; \D s(\vd) \, ,
    \\[1ex]
    \vH^i[\vphi](\vx) & = \sqrt{\frac{\varepsilon}{\mu}} \, \int_{\SS^2} \vd \times \vphi(\vd) \, \E^{\I k \vd \cdot \vx} \; \D s(\vd) \, ,
 \end{aligned}
 \qquad\qquad \vx \in \RR^3\, ,
\end{equation}
as the incident field. By superposition, the corresponding scattered electric field $\vE^s[\vphi]$ is given by
\begin{equation*}
    \vE^s[\vphi](\vx) = \int_{\SS^2} \vE^s(\vx, \vd, \vphi(\vd)) \; \D s(\vd) \, , \qquad \vx \in \RR^3\, ,
\end{equation*}
and its far field pattern is
\begin{equation}
  \label{eq:Einf_vphi}
    \vE^{\infty}[\vphi](\vxhat) = \int_{\SS^2} \vE^{\infty}(\vxhat, \vd, \vphi(\vd)) \; \D s(\vd) \, , \qquad \vx \in \RR^3\, .
  \end{equation}
The far field operator $\mathcal{F}_D \colon L_t^2(\SS^2) \rightarrow L_t^2(\SS^2)$ for the scatterer $D$ is defined by
\begin{align*}
    \mathcal{F}_D \vphi = \vE^{\infty}[\vphi] \, , \qquad \vphi \in L_t^2(\SS^2) \, .
\end{align*}
Note that $\mathcal{F}_D$ is a linear integral operator with analytic kernel. Hence, its singular values are decreasing at least exponentially. In particular, $\mathcal{F}_D$ is  a Hilbert-Schmidt operator and its singular values form a square-summable sequence.
In the \red{S}ections \ref{sec:Derivative}--\ref{sec:NumericalExamples}, we will consider far field operators for many different scatterers $D$, hence we have included the scatterer as an index to the far field operator.

In \cite{FerFruRoc2016}, a definition of an \emph{electromagnetically chiral (em-chiral) scatterer} based on its interaction with electromagnetic fields has been proposed. This notion has been incorporated into the mathematical setting of a scattering problem described above in \cite{AreHag2018}. We here give a short overview over the main results.

Any plane wave can be uniquely decomposed into the sum of two plane waves with circular polarization (i.e. the amplitude vector of the real part of the wave performs a circular motion when observed along any line in the direction of propagation, \red{which is equivalent to $\I \, (\vd \times \vp) = \pm\vp$}) but with opposite orientation. The orientation of this circular motion is called the field's \emph{helicity.} This notion can be generalized by observing that there is a connection to an eigenvalue problem: a \emph{Beltrami field} is an eigenfunction of the operator $(1/k) \Dcurl$ for the eigenvalue $\pm 1$ and the sign of this eigenvalue is called the field's helicity (and this is equivalent to the previous definition in case of a plane wave). It turns out that any solution to the Maxwell system can be uniquely represented as the sum of two Beltrami fields with opposite helicity.

For a Herglotz wave pair, the helicity can be obtained directly from the density $\vphi \in L^2_t(\SS^2)$. Consider the operator $\mathcal{C}\colon L_t^2(\SS^2) \rightarrow L_t^2(\SS^2)$,
\begin{align*}
    \mathcal{C}\vphi(\vd) = \I \left( \vd \times \vphi(\vd) \right) , \qquad \vphi \in L^2_t(\SS^2) \, .
\end{align*}
It is not difficult to establish that the Herglotz wave pair with density $\vphi$ has helicity $p \in \{ +, - \}$ if and only if  $\vphi \in V^p$ where
\begin{equation}
  \label{eq:V_pm}
  V^{\pm} = \left\{ \widetilde{\vphi} \pm \mathcal{C} \widetilde{\vphi} \;\colon\, \widetilde{\vphi} \in L_t^2(\SS^2) \right\} .
\end{equation}
Note also that these two subspaces are orthogonal, that the orthogonal projections $\mathcal{P}^{\pm} \colon L_t^2(\SS^2) \to V^\pm$ are given by $\mathcal{P}^{\pm} = (\mathcal{I} \pm \mathcal{C})/2$, and that $L^2_t(\SS^2) = V^+ \oplus V^-$. Thus, the representation of a Herglotz wave pair with density $\vphi \in L^2_t(\SS^2)$ as a sum of two Beltrami fields of opposite helicity is exactly the same as writing it as the sum of two Herglotz wave pairs with density $\mathcal{P}^\pm \vphi$.

A completely analogous situation holds for a solution $(\vE^s, \vH^s)$ of the Maxwell system in the exterior of some ball that satisfies the Silver--M\"uller radiation condition. The corresponding Beltrami fields of helicity~$\pm$ also satisfy the asymptotics \eqref{eq:Es_asymp} with far field patterns in $V^\pm$, respectively.

In conclusion, we can decompose the far field operator into components associated with particular incident and scattering helicity by composition with the orthogonal projections onto $V^\pm$,
\begin{align*}
    \mathcal{F}_D = \mathcal{F}_D^{++} + \mathcal{F}_D^{+-} + \mathcal{F}_D^{-+} + \mathcal{F}_D^{--}\, ,
    \qquad \mathcal{F}_D^{pq} = \mathcal{P}^{p} \mathcal{F}_D \mathcal{P}^{q} \, , \quad p, q \in \{+,-\} \, .
\end{align*}
The definition of electromagnetic chirality is based on this decomposition.

\begin{definition}\cite{FerFruRoc2016, AreHag2018}
 \label{defi:em-chirality}
 The scatterer $D$ is called electromagnetically achiral (em-achiral) if there exist unitary operators $\mathcal{U}_j \in L^2_t(\SS^2)$ with $\mathcal{U}_j \mathcal{C} = -\mathcal{C} \mathcal{U}_j$, $j = 1,\ldots, 4$, such that
 \begin{equation*}
    \mathcal{F}_D^{++} = \mathcal{U}_1 \mathcal{F}_D^{--} \mathcal{U}_2 \, , \qquad
    \mathcal{F}_D^{+-} = \mathcal{U}_3 \mathcal{F}_D^{-+} \mathcal{U}_4 \, .
 \end{equation*}
 If this is not the case, we call the scatterer D electromagnetically chiral (em-chiral).
\end{definition}

\begin{remark}
\label{remark:chirality}
 \begin{itemize}
  \item[(i)] The condition $\mathcal{U} \mathcal{C} = -\mathcal{C} \mathcal{U}$ for unitary $\mathcal{U}$ means that this operator flips helicity. Such an operator bijectively maps $V^+$ to $V^-$ and vice versa.
  \item[(ii)] As discussed in \cite{AreHag2018}, any perfectly conducting object that can be superimposed onto its mirror image (in other words: a geometrically achiral object), is em-achiral.
  \item[(iii)] Em-achirality implies that the singular values of $\mathcal{F}_D^{++}$ coincide with those of $\mathcal{F}_D^{--}$ and that the singular values of $\mathcal{F}_D^{+-}$ coincide with those of $\mathcal{F}_D^{-+}$.
 \end{itemize}
\end{remark}

At first glance, Definition \ref{defi:em-chirality} provides just a binary criterion: a scatterer is em-chiral or it is not. However, through the connection to the singular values of the component operators, it is possible to introduce a \emph{measure} that quantifies \emph{how em-chiral} a scatterer is. The original measure introduced in \cite{FerFruRoc2016} has turned out not to be sufficiently regular to be suitable as an objective functional in a shape optimization algorithm. Hence, a modified, more regular functional was introduced in \cite{Hagem2019} and has since been successfully used in shape optimization procedures \cite{AreGri2021, FerGri2023}.

The definition of the modified functional requires one additional functional analytic concept: for some Hilbert space $X$ with scalar product $( \cdot , \cdot )_{X}$, we will denote by $\hs(X)$ the space of Hilbert-Schmidt operators on $X$. Additionally, let $\left( g_j \right)_{j\in\NN}$ denote some complete orthogonal system in $X$. Then 
$\hs(X)$  is itself a Hilbert space when equipped with the scalar product (see \cite[XI.6]{DunSch1988} or \cite[Thm VI.22]{ReeSim2010})
\begin{align*}
    \left( \mathcal{G} , \mathcal{H} \right)_{\hs} = \sum_{j=1}^{\infty} \left( \mathcal{G} g_j , \mathcal{H} g_j \right)_X \, , \qquad \mathcal{G}, \mathcal{H} \in \hs(X) \, .
\end{align*}

\begin{definition}
  \label{defi:chi_hs}
    The functional $\chi_{\hs} \colon \hs(L_t^2(\SS^2)) \rightarrow \RR_{\geq 0}$ defined by
    \begin{align*}
        \chi_{\hs}(\mathcal{F}) = \Big( \Vert \mathcal{F} \Vert_{\hs}^2 - 2 \big(\Vert \mathcal{F^{++}} \Vert_{\hs} \Vert \mathcal{F^{--}} \Vert_{\hs} + \Vert \mathcal{F^{+-}} \Vert_{\hs} \Vert \mathcal{F^{-+}} \Vert_{\hs} \big) \Big)^{\mkern-3mu \frac{1}{2}}
    \end{align*}
    is called the \textbf{em-chirality measure}.
\end{definition}

It immediately follows that $\chi_{\hs}(\mathcal{F})$ is bounded by $\Vert \mathcal{F} \Vert_{\hs}$. If a scatterer satisfies
\[
  \mathcal{F^{++}} = 0 \text{ and } \mathcal{F^{-+}} = 0
  \qquad \text{or} \qquad
  \mathcal{F^{--}} = 0 \text{ and } \mathcal{F^{+-}} = 0 \, ,
\]
then \red{this} maximum is attained. Due to reciprocity, which implies $\mathcal{F}^{+-} = 0$ if and only if $\mathcal{F}^{-+} = 0$, the reverse is also true \red{(see \cite[Lmm.~4.3]{AreHag2018})}. A scatterer with this property is called \emph{maximally em-chiral} and it is effectively invisible to incident fields of one of the two helicities.
\red{In Theorem \ref{thm:UniquenessPEC} in the appendix below, we show that the electric far field patterns of the fields of one helicity scattered by any perfect conductor from incident plane waves of the same helicity suffice to uniquely identify this obstacle. In other words, knowledge of~$\mathcal{F}^{++}$ or of~$\mathcal{F}^{--}$ alone suffices to uniquely identify a perfectly conducting obstacle. Since for any maximally em-chiral obstacle these two operators are zero, this shows that a perfectly conducting maximally em-chiral scatterer cannot exist.

\begin{theorem}
  \label{thm:no_max_chiral_object}
  There is no perfectly conducting maximally em-chiral scatterer. 
\end{theorem}}

It is desirable for many applications to make use of scatterers with $\chi_{\hs}$ \red{close to} its maximum, which motivates us to find such objects through a shape optimization process. The idea is to employ a Quasi-Newton scheme to find scatterers with a high value of the ratio $\chi_{\hs}(\mathcal{F}) / \Vert \mathcal{F} \Vert_{\hs}$.

%
\section{The Fr\'echet Derivative of the Far Field Operator}
\label{sec:Derivative}

To perform shape optimization, it is necessary to understand how changes of the shape of $D$ affect the chirality measure $\chi_\hs(\mathcal{F}_D)$. As a first step, in this section we will discuss differentiability of the map~$\partial D \mapsto \mathcal{F}_D$. We will base our analysis on the corresponding derivative of the far field map which, for a fixed incident field $\vE^i$, maps the boundary $\partial D$ to the far field pattern $\vE^\infty$ of the scattered field in the scattering problem \eqref{eq:maxwell}-\eqref{eq:SMRC}. The existence of this Fr\'echet derivative for scattering by a perfect conductor was established in \cite{Hagem2019}. This derivation in turn follows very closely the variational approach taken in \cite{Hettl2012} for a penetrable scatterer. Here, we revisit this derivation and present some modified estimates which turn out to be crucial for establishing Fr\'echet differentiability of $\calF_D$.

For a differentiable vector field $\vh$, we denote its Jacobian by $J_{\vh}$. We equip the space of compactly supported, continuously differentiable vector fields with the norm $\| \vh \|_{1, \infty} = \| \vh \|_\infty + \| J_\vh \|_\infty$, where, for some matrix valued function $J$, $\| J \|_\infty$ denotes the supremum of the row sum norm of $J(\vx)$, $\vx \in \RR^3$. Consider a set
\begin{equation}
 \label{eq:def_P}
 \mathcal{P} = \left\{ \vh \in C^1(\RR^3, \RR^3) : \partial D \subseteq \operatorname{supp} \vh \subseteq \mathcal{U} \text{ with } \mathcal{U} \text{ open, bounded, } \Vert \vh \Vert_{1, \infty} < \frac{1}{2} \right\}
\end{equation}
of admissible perturbations. Then, for $\vh \in \mathcal{P}$, $\vx \mapsto \veta_\vh(\vx) = \vx + \vh(\vx)$ is a diffeomorphism and we obtain the perturbed domain $D_\vh$ with boundary $\partial D_{\vh} = \{ \veta_\vh(\vx) : \vx \in \partial D \}$. The weak formulation of the scattering problem for the perturbed domain is given by \eqref{eq:weakFormOps} with $\vE, \vV$ and $\Omega$ replaced by  $\vE_{\vh}, \vV_{\vh}$ and $\Omega_{\vh} = B_R \setminus \overline{D_{\vh}}$, respectively.

The diffeomorphism defined above makes it possible to pull back to the original geometry. We define
\begin{align*}
  \widehat{\vE}_{\vh}(\vx) = J_{\veta_\vh}^{\top}(\vx) \, \vE_{\vh} (\veta_\vh(\vx)) = \left( I + J_{\vh}^{\top}(\vx) \right) \, \vE_{\vh} (\vx + \vh(\vx)) \, .
\end{align*}
By \cite[Cor. 3.58]{Monk2003}, we know that $\widehat{\vE}_{\vh} \in H(\Dcurl, \Omega)$ if and only if $\vE_\vh \in H(\Dcurl, \Omega_\vh)$ and that
\[
  \left( \Dcurl \vE_\vh \right) ( \vx + \vh(\vx)) = \frac{1}{\det J_{\veta_\vh}(\vx)} \, J_{\veta_\vh}(\vx) \Dcurl \widehat{\vE}_{\vh}(\vx) \, , \quad \vx \in \Omega \, .
\]
We see that $\widehat{\vE}_{\vh}$ satisfies
\begin{align}
  \label{eq:weak_form_ops_with_h}
  \mathcal{A}_{\vh}(\widehat{\vE}_{\vh},\vV) = \ell(\vV) \qquad \text{for all } \vV \in H_{\mathrm{pc}}(\Dcurl, \Omega) \, ,
\end{align}
with
\begin{multline*}
  \mathcal{A}_{\vh}(\widehat{\vE}_{\vh},\vV) = \int_{\Omega} \Big( \Dcurl \widehat{\vE}_{\vh} \cdot \frac{ J_{\veta_\vh}^{\top} J_{\veta_\vh}}{ \det \, J_{\veta_\vh}} \, \overline{\Dcurl \vV} - k^2 \widehat{\vE}_{\vh} \cdot J_{\veta_\vh}^{-1} J_{\veta_\vh}^{-\top} \overline{\vV} \; \det \, J_{\veta_\vh} \Big) \D \vy \\
  {} - \I \omega \mu \int_{\partial B_R} \Lambda( \gamma_t \widehat{\vE}_{\vh}) \cdot \overline{\vV} \; \D s
\end{multline*}

Before considering differentiability, we state the following preliminary continuity result which clarifies the dependence of the estimate of \cite[Thm. 3.3]{Hagem2019} on the incident field.

\begin{lemma}
 \label{lemma:estimates_perturbation_of_E_Hcurl}
  Let $\vh \in \mathcal{P}$ and $\vE$, $\widehat{\vE}_\vh \in H(\Dcurl, \Omega)$ be the solutions of \eqref{eq:weakFormOps} and \eqref{eq:weak_form_ops_with_h} for the same incident field $\vE^i$, respectively. Then,
  \begin{align*}
    \big\Vert \widehat{\vE}_{\vh} - \vE \big\Vert_{H(\Dcurl,\Omega)} \leq C \, \Vert\vh\Vert_{1, \infty} \Vert \vE^i \Vert_{H(\Dcurl,\Omega)} \, .
  \end{align*}
\end{lemma}

\begin{proof}
  We repeat the arguments given in \cite[Thm. 3.3]{Hagem2019} with the necessary modifications to make the dependence on $\vE^i$ explicit. From the Riesz representation theorem we obtain linear bounded and boundedly invertible operators $A$, $A_{\vh} \colon H_{\mathrm{pc}}(\Dcurl, \Omega) \rightarrow H_{\mathrm{pc}}(\Dcurl, \Omega)$ with
  \begin{align*}
    \mathcal{A}(\vE,\vV) = (A \vE, \vV)_{H(\Dcurl,\Omega)} \, ,
    \qquad
    \mathcal{A}_{\vh}(\widehat{\vE}_{\vh}, \vV) = (A_{\vh}\widehat{\vE}_{\vh}, \vV)_{H(\Dcurl,\Omega)} \, .
  \end{align*}
  Let additionally $L \in H(\Dcurl,\Omega)$ be defined such that $\ell(\vV) = (L, \vV)_{H(\Dcurl,\Omega)}$. Then  \eqref{eq:weakFormOps} and \eqref{eq:weak_form_ops_with_h} are equivalent to
  \begin{align*}
    A \vE = L \quad \textnormal{and} \quad A_{\vh}\widehat{\vE}_{\vh} = L \, ,
  \end{align*}
  respectively.

  Exactly as in \cite{Hagem2019} we establish
  \begin{align*}
    \| A_{\vh} - A \| \leq C \, \| \vh \|_{1, \infty}
  \end{align*}
  and  that $\| A_\vh^{-1} \|$ is uniformly bounded.  We conclude
  \begin{align*}
    \| \widehat{\vE}_{\vh} - \vE \|_{H(\Dcurl,\Omega)}
    & = \| A_{\vh}^{-1} A \vE - A_{\vh}^{-1} A_{\vh} \vE \|_{H(\Dcurl, \Omega)}
    \leq \| A_{\vh}^{-1} \| \, \| A - A_{\vh} \| \, \| \vE \|_{H(\Dcurl, \Omega)} \\
    & \leq \| A_{\vh}^{-1} \| \, \| A - A_{\vh} \| \, \| A^{-1} \| \, \| \ell \|_{H_{\mathrm{pc}}(\Dcurl, \Omega)^*}
    \leq C \, \| \vh \|_{1, \infty} \, \| \vE^i \|_{H(\Dcurl, \Omega)} \, .
\end{align*}
\end{proof}

Next, we consider the \textit{material derivative} $\boldsymbol{W} \in H_{\mathrm{pc}}(\Dcurl, \Omega)$ which is defined as the solution of the variational problem
\begin{equation*}
	\mathcal{A}(\boldsymbol{W}, \vV)
	= \int_{\Omega} \Big[ \Dcurl\,{\vE} \cdot \Big( J_{\veta_\vh} + J_{\veta_\vh}^{\top} - \Ddiv \vh \, I \Big) \, \overline{\Dcurl \, \vV}
	- k^2 {\vE} \cdot \big( \Ddiv \vh\, I - J_{\veta_\vh} - J_{\veta_\vh}^{\top} \big) \, \overline{\vV} \Big] \D \vx
\end{equation*}
for all $\vV \in H_{\mathrm{pc}}(\Dcurl, \Omega)$.

\begin{lemma}
 \label{lemma:estimates_dom_der_Hcurl}
  Let $\vh \in \mathcal{P}$ and $\vE$, $\widehat{\vE}_\vh \in H(\Dcurl, \Omega)$ the solutions of \eqref{eq:weakFormOps} and \eqref{eq:weak_form_ops_with_h} for the same incident field $\vE^i$, respectively, and let  $\vW \in H_{\mathrm{pc}}(\Dcurl, \Omega)$ denote the corresponding material derivative. Then
	\begin{align*}
		\big\| \widehat{\vE}_{\vh} - \vE - \boldsymbol{W} \big\|_{H(\Dcurl, \Omega)} \leq C \, \| \vh \|_{1, \infty}^2 \| \vE^i \|_{H(\Dcurl, \Omega)} \, .
	\end{align*}
\end{lemma}

\begin{proof}
  We use the notation from the proof of Lemma \ref{lemma:estimates_perturbation_of_E_Hcurl}. As in the proof of \cite[Thm. 3.4]{Hagem2019} we derive
  \begin{equation*}
    \big| \mathcal{A}(\widehat{\vE}_{\vh} - \vE - \vW, \vV) \big|
    \leq C \, \| \vh \|_{1, \infty} \Big( \| \vh \|_{1, \infty} \, \| \widehat{\vE}_{\vh} \|_{H(\Dcurl, \Omega)}  + \| \widehat{\vE}_{\vh} - \vE \|_{H(\Dcurl, \Omega)} \Big) \| \vV \|_{H(\Dcurl,\Omega)} \, .
  \end{equation*}
  We conclude that
  \begin{align*}
    \| \widehat{\vE}_{\vh} & {} - \vE - \vW \|_{H(\Dcurl,\Omega)}
    \leq \sup_{\vV \in H_{\mathrm{pc}}(\Dcurl, \Omega)} \frac{\big| \mathcal{A}(\widehat{\vE}_{\vh} - \vE - \vW, \vV) \big|}{\| \vV \|_{H(\Dcurl,\Omega)}} \\
    & \leq C \, \| \vh \|_{1, \infty} \Big( \| \vh \|_{1, \infty} \, \| \widehat{\vE}_{\vh} \|_{H(\Dcurl, \Omega)}  + \| \widehat{\vE}_{\vh} - \vE \|_{H(\Dcurl,\Omega)} \Big) \\
    & \leq C \, \| \vh \|_{1, \infty} \Big( \| \vh \|_{1, \infty} \, \| \vE \|_{H(\Dcurl,\Omega)}  + (1 + \| \vh \|_{1, \infty}) \, \| \widehat{\vE}_{\vh} - \vE \|_{H(\Dcurl, \Omega)} \Big) .
  \end{align*}
  Lemma \ref{lemma:estimates_perturbation_of_E_Hcurl} and \eqref{eq:estimate_for_ell} finally yield
  \begin{align*}
    \| \widehat{\vE}_{\vh} & - \vE - \vW \|_{H(\Dcurl,\Omega)} \\
    & \leq C \, \| \vh \|_{1, \infty} \Big( \|\vh\|_{1, \infty} \, \| \vE^i \|_{H(\Dcurl,\Omega)}  + (1 + \| \vh \|_{1, \infty}) \, \|\vh\|_{1, \infty} \, \| \vE^i \|_{H(\Dcurl,\Omega)} \Big) \\
    & \leq C \, \| \vh \|_{1, \infty}^2 \, \| \vE^i \|_{H(\Dcurl,\Omega)} \, .
  \end{align*}
\end{proof}

The material derivative constitutes the linearization of the total field with respect to the perturbation. However, it is not itself a solution to the Maxwell system and that it depends on $\vh$ in $\Omega$, not just on $\vh|_{\partial D}$. To improve the situation, one introduces the \emph{domain derivative} $\vE'$ of $\vE$ by setting
\[
  \vE' = \vW - J_{\vh}^{\top} \vE + J_{\vE} \vh \, .
\]
Henceforward, we will use the subscript ``$\vnu$'' to indicate the normal component of a vector field on $\partial D$, e.g., $\vh_{\vnu} := \vh \cdot \vnu$. The surface gradient on $\partial D$ is denoted by $\DGrad$.

\begin{lemma}
 \label{lemma:E_prime}
 Let $\vh \in \mathcal{P}$. The domain derivative $\vE' \in H(\Dcurl, \Omega)$ is a weak solution of the Maxwell system~\eqref{eq:maxwell} in $\Omega$ that satisfies the boundary condition
 \begin{align}
   \label{eq:domain_derivative}
    \vE' \times \vnu = \I k \vh_{\vnu} \, \gamma_T \vH - \DGrad( \vh_{\vnu} \vE_{\vnu} ) \times \vnu \, .
 \end{align}
 It can be uniquely extended to a radiating solution of the Maxwell system in $\RR^3 \setminus \overline{D}$ and its far field pattern $\left(\vE'\right)^\infty$ satisfies
 \begin{equation}
   \label{eq:domain_derivative_bc}
   \left\| \widehat{\vE}_{\vh}^\infty - \vE^\infty - \left( \vE' \right)^\infty \right\|_{L^2_t(\SS^2)}
   \leq C \, \| \vh \|_{1, \infty}^2 \, \| \vE^i \|_{H(\Dcurl, \Omega)} \, .
 \end{equation}
\end{lemma}

\begin{proof}
 As in the proof of \cite[Thm. 3.6]{Hagem2019} we see that $\vE' \in H(\Dcurl, \Omega)$, that it is indeed a weak solution to the Maxwell system~\eqref{eq:maxwell} and that its tangential trace satisfies the boundary condition \eqref{eq:domain_derivative_bc} and the non-local transparent boundary condition $\Lambda \gamma_t \vE' = \gamma_N \vE'$ on $\partial B_R$. The last observation implies that $\vE'$ can be extended to a radiating solution of the Maxwell system in $\RR^3 \setminus \overline{D}$. Moreover, outside the support of $\vh$, there holds $\vE' = \vW$, so also $\vW$ is a radiating solution to the Maxwell system in $\RR^3 \setminus (\overline D \cup \supp \vh)$ and $\vW^\infty = \left(\vE'\right)^\infty$. We introduce another auxiliary ball $B_{\varrho}$ of radius $\varrho > 0$ centered at zero such that~$(D \cup \supp \vh) \subseteq B_{\varrho} \subseteq B_R$. From the representation formula of the far field pattern as a surface integral, see, e.g., \cite[Thm. 6.24]{ColKre2019}, we obtain
 \begin{align*}
   \Vert\widehat{\vE}_{\vh}^{\infty} - \vE^{\infty} - ( \vE' )^\infty \Vert_{L_t^2(\SS^2)}
   \leq C \Vert \widehat{\vE}_{\vh} - \vE - \vE' \Vert_{H(\Dcurl,B_R \setminus\overline{B_{\varrho}})}
   = C \Vert \widehat{\vE}_{\vh} - \vE - \vW \Vert_{H(\Dcurl,B_R \setminus\overline{B_{\varrho}})} \, .
 \end{align*}
 Hence, with Lemma \ref{lemma:estimates_dom_der_Hcurl},
 \begin{equation*}
   \big\Vert \widehat{\vE}_{\vh}^{\infty} - \vE^{\infty} - \left( \vE' \right)^\infty \big\Vert_{L_t^2(\SS^2)}
   \leq C \big\Vert \widehat{\vE}_{\vh} - \vE - \vW \big\Vert_{H(\Dcurl,\Omega)}
   \leq C \, \| \vh \|_{1,\infty}^2 \, \| \vE^i \|_{H(\Dcurl, \Omega)} \, .
 \end{equation*}
\end{proof}

With these modified estimates for the domain derivative for a single incident field, we can return to the task of establishing Fr\'echet differentiability of the map $\partial D \mapsto \calF_{D}$. Recalling the notation of \eqref{eq:def_herglotz_wave_pair}--\eqref{eq:Einf_vphi}, the natural candidate for this derivative is the operator $\calF_D' : \mathcal{P} \to \hs(L_t^2(\SS^2))$ defined by
\begin{equation}
   \label{eq:def_ffop_frechet}
  \left( \calF_{D}' \vh \right) \vphi = (\vE')^\infty[\vphi] \, , \qquad \vphi \in L^2_t(\SS^2) \, .
\end{equation}
The main result of this section -- that we will prove in its remainder -- is the following theorem stating that $\calF_{D}'$ is indeed the Fr\'echet derivative of $\partial D \mapsto \calF_{D}$.

\begin{theorem}
  \label{thm:ffop_frechet}
  The operator $\calF'_{D}  : \mathcal{P} \to \hs(L_t^2(\SS^2))$ given by \eqref{eq:def_ffop_frechet} satisfies
  \begin{align}
    \label{eq:frechet_der_ffop}
    \| \calF_{D_{\vh}} - \calF_{D} - \calF'_{D} \vh \|_{\hs}
    \leq C \, \| \vh \|_{1, \infty}^2 \, ,
  \end{align}
  and hence is the Fr\'echet derivative of $\calF_{D}$.
\end{theorem}

The proof of this theorem requires a number of steps that will be formulated as some intermediate lemmas below. As the first step, we reduce the assertion to an estimate over a series of norms of electric Herglotz wave functions. We remind the reader that the Hilbert-Schmidt norm on $\hs(L^2_t(\SS^2))$ can be expressed via any complete orthonormal system in $L^2_t(\SS^2)$. Here, we use the vector spherical harmonics. Denote the spherical harmonics, defined as in \cite[Equ. (2.28)]{ColKre2019}, by $Y_n^m$. The vector spherical harmonics are defined by
\begin{equation}
  \label{eq:vec_sph_harmonics}
  \vU_n^m(\vxhat) = \frac{1}{\sqrt{n(n+1)}} \, \DGrad Y_n^m(\vxhat) \, , \qquad
  \vV_n^m(\vxhat) = \vxhat \times \vU_n^m(\vxhat) \, ,
  \qquad \vxhat \in \SS^2 \, ,
\end{equation}
with $n\in\NN$, $m = -n, \ldots, n$. The completeness property is shown in \cite[Lmm. 9.15]{Monk2003}, for example. By 
\cite[XI.6]{DunSch1988}, the norm of a Hilbert-Schmidt operator $T \in \hs(L_t^2(\SS^2))$ is then given by
\begin{equation}
  \label{eq:hs_norm_by_ons}
  \Vert T \Vert_{\hs} = \left( \sum_{n=1}^{\infty} \sum_{m=-n}^n \| T \vU_n^m \|_{L_t^2(\SS^2)}^2 + \| T \vV_n^m \|_{L_t^2(\SS^2)}^2 \right)^{1/2} \, .
\end{equation}

\begin{lemma}
  \label{lemma:HS_norm_estimate}
  For $\vh \in \mathcal{P}$,
  \begin{equation*}
    \| \calF_{D_{\vh}} - \calF_{D} - \calF'_{D} \vh \|_\hs
    \leq C \,  \| \vh \|_{1, \infty}^{\red{2}} \left( \sum_{n=1}^{\infty}\sum_{m=-n}^{n} \| \vE^i[\vU_n^m] \|_{H(\Dcurl, B_R)}^2 + \| \vE^i[\vV_n^m] \|_{H(\Dcurl, B_R)}^2 \right)^{1/2} .
  \end{equation*}
\end{lemma}

\begin{proof}
	The assertion follows immediately from \eqref{eq:hs_norm_by_ons} and Lemma \ref{lemma:E_prime}.
\end{proof}

It can be readily calculated that
\begin{equation*}
    \Dcurl \vE^i[\vU_n^m] = \I k \vE^i[\vV_n^m] \, , \qquad
    \Dcurl \vE^i[\vV_n^m] = -\I k \vE^i[\vU_n^m] \, ,
\end{equation*}
and hence it remains to show convergence of
\begin{equation*}
    \sum_{n=1}^{\infty}\sum_{m=-n}^{n} \Vert \vE^i[\vU_n^m] \Vert_{L^2(B_R)}^2 \qquad \text{ and } \qquad \sum_{n=1}^{\infty}\sum_{m=-n}^{n} \Vert \vE^i[\vV_n^m] \Vert_{L^2(B_R)}^2 \, .
\end{equation*}
We will achieve this by explicitly calculating these norms through expansions in vector wave functions. In what follows, we will always represent $\vx \in \RR^3$ as $\vx = r \widehat{\vx}$ with $\widehat{\vx} \in \SS^2$ and $r > 0$. Also set
\[
  \vW_n^m (\widehat{\vx}) = Y_n^m(\vxhat) \vxhat \, , \qquad \widehat{\vx} \in \SS^2 \, ,
\]
and let $j_n$ denote the spherical Bessel function of order $n$.

\begin{lemma}
  \label{lemma:Ei_Unm_Vnm}
  For any $\vx \in \RR^3$, there hold
  \begin{align*}
    \vE^i[\vU_n^m](\vx) & = 4\pi\I^n \, \frac{1}{\I k \, | \vx |} \left( \sqrt{n(n+1)} \, j_n(k | \vx |) \, \vW_n^m (\widehat{\vx}) + \big( j_n(k | \vx |) + k | \vx| \,  j_n'(k | \vx |) \big) \, \vU_n^m(\vxhat) \right) , \\
    \vE^i[\vV_n^m](\vx) & = 4\pi \I^n \, j_n(k | \vx |) \, \vV_n^m(\vxhat) \, .
  \end{align*}
\end{lemma}

\begin{proof}
  In addition to $j_n$, we require the spherical Hankel function $h_n^{(1)}$ of order $n$ of the first kind. For~$n \in \NN_0$, $m = -n, \ldots, n$, let
  \begin{equation*}
    u_n^m(\vx) = j_n(k | \vx |) \, Y_n^m(\vxhat), \qquad   \vM_n^m(\vx) = -j_n(k\vert \vx \vert) \, \vV_n^m(\vxhat), \qquad \vx \in \RR^3 \, ,
  \end{equation*}
  and
  \begin{equation*}
    v_n^m(\vx) = h_n^{(1)}(k | \vx |) \, Y_n^m(\widehat{\vx}), \qquad \vN_n^m(\vx) = -h_n^{(1)}(k\vert \vx \vert) \, \vV_n^m(\vxhat) \, , \qquad \vx \in \RR^3 \setminus \{ 0 \} \, .
  \end{equation*}
  With these functions and the vector addition theorem in the form of \cite[Equ. (6.83)]{ColKre2019}, we obtain for any~$\vA \in \CC^3$,
  \begin{multline}
      \label{vector_addition_theorem}
      \Phi(\vx,\boldsymbol{y}) \vA
      = \sum_{n=1}^{\infty} \sum_{m=-n}^n \Big[ \I k  \left( \vN_n^m(\vy) \cdot \vA \right) \overline{\vM_n^m (\vx)}
      + \frac{\I}{k}  \left( \Dcurl \vN_n^m(\vy) \cdot \vA \right) \overline{\Dcurl \vM_n^m(\vx)} \\
      {} + \frac{\I}{k} \left( \nabla v_n^m(\vy)  \cdot \vA \right) \overline{\nabla u_n^m(\vx)} \Big] \, ,
  \end{multline}
  where
  \begin{align*}
  	{\Phi}(\vx,\vy) = \frac{\E^{\I k \vert \vx - \vy \vert}}{4\pi \vert \vx - \vy \vert} \, , \qquad \vx, \vy \in \RR^3 \, , \quad \vx \neq \vy \, ,
  \end{align*}
  denotes the fundamental solution to the Helmholtz equation.
  For fixed $\vx \in \RR^3$, the left hand side of \eqref{vector_addition_theorem} is a radiating vector solution to the Helmholtz equation with far field pattern
  \[
    \left( \Phi(\vx, \cdot) \vA \right)^\infty (\vyhat)
    = 4 \pi \, \lim_{r \to \infty} \left( r \, \E^{-\I k r} \Phi(\vx, r \vyhat) \vA \right)
    = \vA  \, \E^{- \I k \vx \cdot \vyhat} \, .
  \]
  We also compute the far field patterns of the right hand side of \eqref{vector_addition_theorem}. As in the proof of \cite[Thm. 6.28]{ColKre2019}, we obtain
  \begin{align*}
    \left( \vN_n^m \right)^\infty = -\frac{4\pi}{\I^{n+1} \, k} \, \vV_n^m
    \qquad \text{ and } \qquad
    (\Dcurl\,\vN_n^m)^{\infty} = \frac{4\pi}{\I^{n}} \, \vU_n^m \, ,
  \end{align*}
  respectively. Also, writing the gradient in spherical coordinates, using a recursion formula for the derivative of the spherical Bessel functions and noting $h_n^{(1)}(kr) \to 0$, $kr \, h_n^{(1)}(kr)  \to (-\I)^{n+1}$ $(r \to \infty)$,
  \begin{align*}
   \left( \nabla v_n^m \right)^\infty(\vyhat)
   & = 4\pi \, \lim_{r \to \infty}  \E^{-\I k r} \, r \left( \frac{\partial v_n^m(\vy)}{\partial r} \vyhat + \frac{1}{r} \, \DGrad_{\SS^2}  v_n^m(\vy) \right) \\
   & = 4\pi \, \lim_{r \to \infty}  \E^{-\I k r} \left( kr \, {h_n^{(1)}}'(kr) \vW_n^m(\vyhat) + \sqrt{n (n+1)} \, h_n^{(1)}(kr) \, \vU_n^m(\vyhat) \right) \\
   & = 4 \pi \, \lim_{r \to \infty}  \E^{-\I k r}  \left( -kr \,  h_{n + 1}^{(1)}(k r) - n \, h_n^{(1)}(k r) \right) \vW_n^m(\vyhat)
   = \frac{4 \pi}{\I^{n}} \, \vW_n^m(\vyhat) \, .
  \end{align*}
  Thus, we obtain
  \begin{multline*}
      \vA  \, \E^{- \I k \vx \cdot \vyhat}
      = 4\pi \sum_{n=1}^{\infty} \sum_{m=-n}^n \frac{1}{\I^n} \, \Big[ - \left( \vV_n^m(\vyhat) \cdot \vA \right) \overline{\vM_n^m (\vx)}
      + \left( \vU_n^m(\vyhat) \cdot \vA \right) \overline{ \frac{1}{\I k} \, \Dcurl \vM_n^m(\vx)} \\
      {} + \frac{\I}{k} \left( \vW_n^m(\widehat{\vy})  \cdot \vA \right) \overline{\nabla u_n^m(\vx)} \Big] \, .
  \end{multline*}
  We now substitute $\vA = \overline{\vphi(\vd)}$, $\vphi \in L_t^2(\SS^2)$, and $\vyhat = \vd$ and take the complex conjugate of the equation to obtain
  \begin{equation*}
      \vphi(\vd)  \, \E^{\I k \vx \cdot \vd}
      = 4 \pi \sum_{n=1}^{\infty} \sum_{m=-n}^n \I^n \, \bigg[ - \left( \vphi(\vd) \cdot \vV_n^{-m}(\vd) \right) \vM_n^m (\vx)
      + \left( \vphi(\vd) \cdot \vU_n^{-m}(\vd) \right) \frac{1}{\I k} \, \Dcurl \vM_n^m(\vx) \bigg] \, .
  \end{equation*}
  Substituting $\vphi = \vU_n^m$ or $\vphi = \vV_n^m$ and integrating over $\SS^2$ with respect to $\vd$ lets us conclude
  \begin{equation}
    \label{Ei_with_densities_Unm_Vnm}
      \vE^i[\vU_n^m](\vx) = 4\pi \I^n \, \frac{1}{\I k} \Dcurl\,\vM_n^m(\vx) \, , \qquad
      \vE^i[\vV_n^m](\vx) = -4\pi \I^n \, \vM_n^m(\vx) \, .
  \end{equation}

  For $\vx \in \RR^3 \setminus \{0\}$ we calculate as in the proof of \cite[Thm. 2.43]{KirHet2015}
  \begin{equation*}
    \Dcurl \vM_n^m(\vx) = \frac{ \sqrt{n(n+1)} }{ | \vx | } \, j_n(k |\vx| ) \, \vW_n^m(\widehat{\vx}) + \frac{1}{|\vx|} \, \big( j_n(k | \vx |)
    + k | \vx | \, j_n'(k | \vx | ) \big) \, \vU_n^m(\widehat{\vx}) ) \, .
  \end{equation*}
  Thus, the assertion follows.
\end{proof}

\begin{lemma}
 \label{lemma:Ei_Unm_Vnm_norms_as_integrals}
 The $L^2$-norms of the Herglotz wave functions with vector spherical harmonics as densities are given by
 \begin{align*}
  \| \vE^i[\vU_n^m] \|_{L^2(B_R)}^2 &= 16 \pi^2 \frac{n(n+1)}{k^2} \int_0^R \vert j_{n}(k r) \vert^2 \, \D r + 16 \pi^2 \frac{1}{k^2} \int_{0}^R \big\vert \big( r j_n(kr) \big)' \big\vert^2 \, \D r \, , \\
  \| \vE^i[\vV_n^m] \|_{L^2(B_R)}^2 &=  16\pi^2 \int_0^R \vert r j_n(kr) \vert^2 \, \D r \, .
 \end{align*}
\end{lemma}

\begin{proof}
  The assertion follows from Lemma \ref{lemma:Ei_Unm_Vnm} and the fact that the functions $\vU_n^m$, $\vV_n^m$ and $\vW_n^m$ form an orthonormal system in $L^2(\SS^2)$ (see \cite[Thm. 5.36]{KirHet2015}).
\end{proof}

\begin{lemma}
 \label{lemma:bessel_series}
 For any $\xi \in \NN$, the series
 \[
  \left( \sum_{n=0}^{\infty} n^{\xi} \, | j_{n}(z) |^2 \right) \qquad \text{and} \qquad \left( \sum_{n=0}^{\infty}  n^{\xi} \, | j_{n}'(z) |^2\right)
 \]
 converge uniformly with respect to $z$ on any compact subset of $\CC$.
\end{lemma}

\begin{proof}
  By \cite[Thm. 2.31]{KirHet2015}, we have
  \[
     n^\xi \, | j_n(z) |^2 = \frac{n^\xi}{(2n + 1)!!} \, \frac{| z |^{2n}}{(2n + 1)!!} \left( 1 + \mathrm{O}\left( \frac{1}{n} \right) \right), \qquad n \to \infty \, ,
  \]
  uniformly on compact subsets of $\CC$, where $(2n+1)!! = 1 \cdot 3 \cdots (2n+1)$ for $n \in \NN$. Thus, by the ratio test, the first series converges uniformly on compact subsets of $\CC$.

 The assertion for the second series follows from that of the first also using the recurrence relation \cite[Equ. 10.51.1]{NIST}
 \[
   j_n'(z) = \frac{n}{2n+1} \, j_{n-1}(z) - \frac{n+1}{2n+1} \, j_{n+1}(z) \, .
 \]
\end{proof}

\begin{proof}[Proof of Theorem~\ref{thm:ffop_frechet}]
  We have established that
  \begin{align*}
    \sum_{n=1}^{\infty}\sum_{m=-n}^{n} \Vert \vE^i[\vU_n^m] \Vert_{L^2(B_R)}^2
    & = \frac{16 \pi^2}{k^2} \int_0^R \left[ \sum_{n=1}^\infty (2n + 1) \left( n(n+1) \, | j_n(kr) |^2 + \left| \left( r j_n(kr) \right)' \right|^2 \right) \right] \D r \, , \\
    \sum_{n=1}^{\infty}\sum_{m=-n}^{n} \Vert \vE^i[\vV_n^m] \Vert_{L^2(B_R)}^2
    & = 16 \pi^2 \int_0^R \left[ \sum_{n=1}^\infty (2n + 1) \left| r j_n(kr) \right|^2 \right] \D r \, ,
  \end{align*}
  provided the series in the integrals converge uniformly for $r \in [0, R]$. Once we have shown that this is indeed the case, the assertion of the theorem follows from \red{Lemma} \ref{lemma:HS_norm_estimate}.

  Elementary estimates yield
  \begin{gather*}
  \begin{aligned}
   (2n + 1) \, n(n+1) \, | j_n(kr) |^2 &\leq 6 n^3 \, | j_n(kr) |^2 \, , \\[1ex]
   (2n + 1) \left| \left( r j_n(kr) \right)' \right|^2
   &= (2n + 1) \left| j_n(kr) + k r \, j_n'(kr) \right|^2
   \leq 6n \, | j_n(kr) |^2 + 6n k^2 R^2 \, | j_n'(kr) |^2 \, , \\[1ex]
   (2n + 1) \left| r j_n(kr) \right|^2 &\leq 3n R^2 \, | j_n(kr) |^2 \, .
   \end{aligned}
  \end{gather*}
  The uniform convergence on $[0, R]$ of the series over the right hand sides in all three estimates has been established in Lemma \ref{lemma:bessel_series}. This completes the proof of the theorem.
\end{proof}

\section{The Optimal Design Algorithm}
\label{sec:OptDesignAlgorithm}

In this section, we explain our approach to optimize the shape of obstacles from the class of tubular domains, which is described in detail in \cite{AreKnoSch2026}. We will give specifics below, but for the moment it suffices to say that for some open subset $\mathcal{M}$ of a normed space $\mathcal{Y}$, every element of $\vX \in \mathcal{M}$ corresponds to one admissible domain $D[\vX]$. Our algorithm aims to maximize the normalized chirality measure $\mathcal{J}_{\hs} : \mathcal{M} \to [0, 1]$ defined by
\begin{equation}
  \label{eq:J_hs}
  \mathcal{J}_{\hs}(\vX) = \frac{\chi_{\hs}(\calF_{D[\vX]})}{\| \calF_{D[\vX]} \|_{\hs}} \, , \quad \vX \in \mathcal{M} \, .
\end{equation}
Due to the ill-posed nature of the shape optimization problem, we have to add four regularization terms~$\Psi_1, \ldots, \Psi_4$, yielding the optimization problem
\begin{equation}
 \label{eq:reg_optim_prob}
 \text{find } \quad \underset{\vX \in \mathcal{M}}{\operatorname{argmin}} \, \big[ - \Upsilon(\vX) \big]
 \qquad \text{with} \quad
 \Upsilon(\vX) = \mathcal{J}_{\hs}(\vX) + \sum_{j=1}^4 \alpha_j \, \Psi_j(\vX) \, , \quad \vX \in \mathcal{M} \, .
\end{equation}

Below, we establish that all terms in the objective functional $\Upsilon$ are Fr\'echet differentiable and derive explicit formulas for evaluating these derivatives. Hence, we can apply a BFGS scheme with a cautious update rule and an inexact Armijo-type line search as suggested in \cite{LiFuk2001} and previously used in \cite{AreGri2021, FerGri2023, Knoel2023}. Fix a subspace $\mathcal{Y}_m \subseteq \mathcal{Y}$ of dimension $m$ and let $\mathcal{M}_m = \mathcal{M} \cap \mathcal{Y}_m$. In the BFGS scheme we construct a sequence $(\vXj)$ from $\mathcal{M}_m$. In each iteration step, we need to compute the gradient $\nabla \Upsilon(\vXj)$ in the topology of $\mathcal{Y}_m$ and to approximate the Hessian of $\Upsilon$. Denote by $\vHj$ the $j$th iterate of some symmetric, positive definite matrix approximating the true Hessian. Then, we consider the convex quadratic approximation of $\Upsilon$ at the current iterate $\vXj$
\begin{equation*}
	q_j(\vd_{\text{BFGS}}) = \Upsilon(\vXj) + \nabla \Upsilon(\vXj) \cdot \vd_{\text{BFGS}} + \frac{1}{2} \, \vd_{\text{BFGS}} \cdot \vHj \vd_{\text{BFGS}}
\end{equation*}
which is minimized by
\begin{equation*}
	\vd_{\text{BFGS}, j} = - \vHj^{-1} \, \nabla \Upsilon(\vXj) \, .
\end{equation*}
We use this minimizer as the search direction, getting the $(j+1)$-th iterate
\begin{equation*}
	\vXjj = \vXj + \lambda_j \vd_{\text{BFGS},j} \, .
\end{equation*}
For the cautious update rule of $\vHj$ we define
\begin{equation*}
	\vSj = \vXjj - \vXj = \lambda_j \vd_{\text{BFGS}, j} \, , \quad \vYj = \nabla \Upsilon(\vXjj) - \nabla \Upsilon(\vXj)
\end{equation*}
and update $\vHj$ through
\begin{equation*}
	\vHjj = \vHj - \frac{\vHj \vYj \vYj^{\top} \vHj}{\vYj^{\top} \vHj \vYj} + \frac{\vSj \vSj^{\top}}{\vYj^{\top} \vSj} \qquad \text{if } \frac{\vYj^{\top} \vSj}{|\vSj|^2} > \varepsilon_{\text{BFGS}} \, |\nabla \Upsilon(\vXj)|
\end{equation*}
for some $\varepsilon_{\text{BFGS}}$. Hence, we perform the standard BFGS update, if a stronger curvature condition than $\vYj^\top \vSj > 0$ is satisfied. We set $\vHjj = \vHj$ otherwise. This ensures positive definiteness of $\vHj$ for any $j \in \NN_0$.

In the inexact Armijo-type line search for determining the step size $\lambda_j$ as proposed in \cite{LiFuk2001}, for given $\sigma \in (0, 1)$, $\delta \in (0,1)$, we want to find the smallest $m \in \NN_0$ such that $\delta^m$ satisfies
\begin{align*}
	\Upsilon(\vXj + \delta^m \vd_{\text{BFGS}, j}) \leq \Upsilon(\vXj) + \sigma \delta^m \, \nabla \Upsilon(\vXj) \cdot \vd_{\text{BFGS}, j} \, .
\end{align*}
We then set $\lambda_j = \delta^m$.

The main numerical effort in this scheme is spent on the computation of the gradient of $\Upsilon$. We have to establish differentiability and state computable formulae for the derivatives. We first focus on $\mathcal{J}_{\hs}$ and address the regularization terms later. The em-chirality measure $\chi_\hs$ was defined in Definition \ref{defi:chi_hs}. A formula for its derivative has been given in \cite{AreGri2021, FerGri2023} and can easily be derived from \cite[Lmm. 5.15]{Hagem2019}. We here state it in the following lemma:

\begin{lemma}
  \label{lemma:deriv_chi}
   Let
	\begin{equation*}
		\mathcal{D} = \Big\{ \mathcal{G} \in \hs(L_t^2(\SS^2)) \;\colon \chi_{\hs}(\mathcal{G}) \neq 0,~ \Vert \mathcal{G}^{pq} \Vert_{\hs} \neq 0, ~p,q \in \{ +, - \} \Big\}
	\end{equation*}
  and assume $\calF \in \mathcal{D}$. Then $\chi_{\hs}$ is Fr\'echet differentiable at $\calF$ with
 \begin{equation}
   \label{eq:deriv_chi}
    \chi_{\hs}'[\calF]\mathcal{H} = \frac{1}{\chi_{\hs}(\calF)} \left( \real\left( \calF, \mathcal{H} \right)_{\hs} - \sum_{p,q\in\{+,-\}} \real\left( \calF^{pq}, \mathcal{H}^{pq} \right)_{\hs} \frac{\Vert \calF^{(-p)(-q)} \Vert_{\hs}}{\Vert \calF^{pq} \Vert_{\hs}} \right)
  \end{equation}
  for all $\mathcal{H} \in \hs(L_t^2(\SS^2))$.
\end{lemma}

With this lemma and our results on Fr\'echet differentiability of the far field operator in Theorem \ref{thm:ffop_frechet}, it is now straightforward to establish differentiability of $\mathcal{J}_\hs$ from \eqref{eq:J_hs} in any $\vX \in \mathcal{M}$. For the statement of the theorem, recall the definition of the set $\mathcal{P}$ in \eqref{eq:def_P} and of the perturbed domain $D_\vh$ for $\vh \in \mathcal{P}$.

\begin{theorem}
 \label{theo:Jhs_frechet_deriv}
 Let $\vX \in \mathcal{M}$ and $K_\rho \subseteq \mathcal{M}$ the open ball of radius $\rho$ centered at the origin such that $\vX + K_\rho \subseteq \mathcal{M}$. Let $\mathcal{H} \colon K_\rho(\vX) \to \mathcal{P}$ such that $D[\vX + \delta\vX] = D_{\mathcal{H}(\delta\vX)}$, $\delta \vX \in K_\rho$. Further assume that~$\mathcal{H}$ is Fr\'echet differentiable at $0$. Then the normalized chirality measure $\mathcal{J}_{\hs}$ is Fr\'echet differentiable at $\vX$. With
 \[
   \vh = \D \mathcal{H}[0] \, \delta\vX \, , \qquad \delta\vX \in \mathcal{Y} \, ,
 \]
 its Fr\'echet derivative is given by
  \[
   \D \mathcal{J}_{\hs}[\vX] \, \delta \vX
   = \frac{ \chi_{\hs}'[\mathcal{F}_{D[\vX]}] \, \mathcal{F}_{D[\vX]}' \vh }{ \| \mathcal{F}_{D[\vX]} \|_{\hs} } - \frac{ \chi_{\hs} (\mathcal{F}_{D[\vX]}) \, \real \langle \mathcal{F}_{D[\vX]}, \mathcal{F}_{D[\vX]}' \vh \rangle_{\hs} }{ \| \mathcal{F}_{D[\vx]} \|_{\hs}^3 } \, , \qquad \delta\vX \in \mathcal{Y} \, ,
  \]
  with $\chi'_{\hs}$ given by \eqref{eq:deriv_chi} and $\mathcal{F}_{D[\vX]}' \vh$ defined by \eqref{eq:def_ffop_frechet}.
\end{theorem}

\begin{proof}
  By Theorem \ref{thm:ffop_frechet}, $\mathcal{F}_D$ is Fr\'echet differentiable at $D[\vX]$ and by the chain rule we have
  \[
   \D \left( \vX \mapsto \mathcal{F}_{D[\vX]} \right) [\vX] = \mathcal{F}_{D[\vX]}' \D \mathcal{H}[0] \, .
  \]
  Recall that $\chi_{\hs}$ is Fr\'echet differentiable by Lemma \ref{lemma:deriv_chi}. Hence, a further application of the chain rule gives
  \[
   \D \left( \vX \mapsto \chi_{\hs} (\mathcal{F}_{D[\vX]}) \right) [\vX] = \chi_{\hs}'[\mathcal{F}_{D[\vX]}] \, \mathcal{F}_{D[\vX]}' \D \mathcal{H}[0] \, .
  \]
  This shows that $\mathcal{J}_{\hs}$ is also Fr\'echet differentiable at $\vX$ and the final formula follows from standard differentiation rules.
\end{proof}

The domains used in our shape optimization scheme have tubular shapes and are represented by a regular twice continuously differentiable spine curve $\vz \colon [0,1] \to \RR^3$ with a continuous accompanying orthogonal frame $(\vt, \vn, \vb)$ and by a continuously differentiable radius function $r \colon [0,1] \to \RR_{> 0}$. The body of the tube is given by the parametrization
\begin{equation}
 \label{eq:tubelike_param}
  \begin{aligned}
    \vx(\tau, \varphi) & = \vz(\tau) + r(\tau) \, \vzeta(\tau, \varphi) \, , \\
    \text{where} & \ \vzeta(\tau, \varphi) = \cos ( \varphi ) \, \vn(\tau) + \sin( \varphi ) \, \vb(\tau) \, ,
  \end{aligned}  \qquad
  \tau \in (0,1) \, , \quad \varphi \in (-\pi, \pi] \, .
\end{equation}
Caps are added at both ends of the tube, so that an overall $C^1$ regular domain is obtained. Thus, the normed space $\mathcal{Y}$ is given by $\mathcal{Y} = C^2([0,1]) \times C^1([0, 1])$ and the set $\mathcal{M}$ is obtained by restricting to those functions for which well-defined tubular domains are obtained from \eqref{eq:tubelike_param}. The precise construction is carried out in~\cite{AreKnoSch2026}. For this class of domains, the proof of Theorem 3.5 in \cite{AreKnoSch2026} moreover contains the construction of the map $\mathcal{H}$ in Theorem \ref{theo:Jhs_frechet_deriv} and explicit computable expressions for $\vh$ are also given.

To obtain the finite dimensional subspace $\mathcal{Y}_m$, we restrict $\vz$ to the space of clamped B-spline curves of polynomial degree $3$ with $n+1$ control points $\vc^{(0)}, \ldots, \vc^{(n)} \in \RR^3$. We select $m = n - 1$ distinct knots distributed uniformly in $[0,1]$, so that the overall regularity of $\vz$ is $C^2$. For $r$, we restrict to the space of natural cubic splines with the knots of the B-splines as interpolation points. The accompanying orthogonal frame is obtained by constructing a rotation minimizing frame with the algorithm provided in \cite{WanJue2008}.

The computation of $\nabla \Upsilon(\vX_j)$ for some iterate $\vX_j \in \mathcal{M}_j$ requires the solution of an exterior Maxwell problem \eqref{eq:maxwell}--\eqref{eq:SMRC} to compute the total fields for the domain represented by $\vX_j$ as well as one such problem with the boundary condition \eqref{eq:domain_derivative} with a corresponding $\vh$ for each element of a basis of $\mathcal{Y}_m$.

Numerically, we solve the problems through a boundary integral formulation using the electric field integral equation (EFIE).
In what follows, the superscripts ``$+$'' and ``$-$'' for $\gamma_t$ indicate traces onto the boundary from out- and inside, respectively.
We define the \textit{electric Maxwell single layer potential} $\boldsymbol{\Psi}_{\mathrm{SL}} \colon H^{-1/2}(\DDiv, \partial D) \rightarrow H_{\text{loc}}(\Dcurl^2, \RR^3\setminus\overline{D})$ by
\begin{align*}
	\boldsymbol{\Psi}_{\mathrm{SL}}\vphi(\vx) = \I k \int_{\partial D} \vphi(\vy) {\Phi}(\vx, \vy) \, \D s(\vy) - \frac{1}{\I k} ~ \nabla \int_{\partial D} \DDiv\vphi(\vy) \, {\Phi}(\vx, \vy) \, \D s(\vy) \, , \qquad \vx \notin \partial D \, ,
\end{align*}
where, again, $\Phi$ denotes the fundamental solution to the Helmholtz equation. We then use the ansatz
\begin{align}\label{eq:ansatz efie}
	\vE^s = - \boldsymbol{\Psi}_{\mathrm{SL}} \vlambda \quad \text{or} \quad \vE' = - \boldsymbol{\Psi}_{\mathrm{SL}} \vlambda \qquad \textnormal{in } \RR^3\setminus\overline{D} \, ,
\end{align}
with the unknown $\vlambda \in H^{-1/2}(\DDiv, \partial D)$.
Continuity of $\boldsymbol{\Psi}_{\mathrm{SL}}$ 
has been shown in \cite[Sec. 3]{BufHip2001}.
Furthermore, we define the \textit{electric Maxwell single layer operator} $\boldsymbol{S} \colon H^{-\frac12}(\DDiv, \partial D) \rightarrow H^{-1/2}(\DDiv, \partial D)$ by
\begin{align*}
	\boldsymbol{S} = \frac12 \left( \gamma_t^+ \boldsymbol{\Psi}_{\mathrm{SL}} + \gamma_t^- \boldsymbol{\Psi}_{\mathrm{SL}} \right) .
\end{align*}
Continuity of $\boldsymbol{S}$ is shown in \cite[Cor. 3.4]{BufHip2004}.
Applying $\gamma_t$ to \eqref{eq:ansatz efie}, exploiting the perfect conductor boundary condition \eqref{eq:bc_perf_cond} and using the jump relations for the single layer potential \cite[Sec. 3]{BufHip2001}, we arrive at the EFIE
\begin{align}
	\label{eq:efie}
	\boldsymbol{S} \vlambda =
	\begin{cases}
		\gamma_t^+ \vE^i \, , & \quad \text{for the scattering problem} \, , \\
		-\I k \vh_{\vnu} \, \gamma_T \vH + \DGrad( \vh_{\vnu} \vE_{\vnu} ) \times \vnu \, , & \quad \text{for the domain derivative} \, .
	\end{cases}
\end{align}
By \cite[Thm. 10]{BufHip2003}, we know that \eqref{eq:efie} admits a unique solution $\vlambda \in H^{-1/2}(\DDiv, \partial D)$ if $k^2$ is not an interior eigenvalue of the $\Dcurl\Dcurl$-operator in $D$ for the homogeneous perfect conductor boundary condition.
The corresponding far field pattern is then obtained by
\begin{align*}
	\vE^{\infty}(\vxhat) = - \boldsymbol{\Psi}_{\mathrm{SL}}^{\infty} \vlambda (\vxhat) \, , \qquad \vxhat \in \SS^2 \, ,
\end{align*}
where for $\boldsymbol{\Psi}_{\mathrm{SL}}^{\infty}$ we just need to replace the fundamental solution with its far field pattern in the definition of $\boldsymbol{\Psi}_{\mathrm{SL}}$.

It remains to discuss the four regularization terms $\Psi_j$, $j \in \{1, 2, 3, 4\}$, present in \eqref{eq:reg_optim_prob}. The first regularization term
\begin{align*}
	\Psi_1 (\vz) = \int_0^1 \kappa^2(\tau) |\vz'(\tau)| \; \D \tau \, , \qquad \text{with curvature } \kappa(\tau) = \frac{|\vz'(\tau) \times \vz''(\tau)|}{|\vz'(\tau)|^3} \, ,
\end{align*}
regularizes the curvature of the spine curve the tube is based on. The second regularization term
\begin{align*}
	\Psi_2 (\vP) = \sum_{j = 0}^{n-1} \left( \left( \frac{1}{n-1} \sum_{i = 0}^{n - 1} \left| \vx^{(i+1)} - \vx^{(i)} \right| \right) - \left| \vx^{(j+1)} - \vx^{(j)} \right| \right)^2  \, ,
\end{align*}
where $\vP$ is a collection of $n+1$ control points defining the clamped B-spline that parameterizes $\vz$, promotes a uniform distribution of these points.
Note, that this term differs from the second regularization term in \cite{AreKnoSch2026}.
The third term
\begin{align*}
	\Psi_3 (r) = \int_0^1 | r'(\tau) |^2 \; \D \tau
\end{align*}
penalizes large radius variations.
Finally, we add a regularization term, which penalizes too small or negative radii.
To this end, we specify a minimal radius value $r_{\min}$, below which the radius function
should not fall, and some positive value $l < r_{\min}$.
Then, we define the function
\begin{align*}
	f_{r_{\min}}^l(r(t)) =
	\begin{cases}
		\dfrac{1}{2 \pi l} \, , & \quad r(t) \in (-\infty, r_{\min} - l] \, , \\[2.5ex]
		\dfrac{1}{2 \pi l} \, \exp\left({-\dfrac{(r(t) + l - r_{\min})^2}{2 \, l^2}}\right) , & \quad  r(t) \in (r_{\min} - l, \infty) \, ,
	\end{cases}
\end{align*}
and with that the fourth term
\begin{align*}
	\Psi_4 (r) = \int_0^1 f_{r_{\min}}^l(r(\tau)) \; \D \tau \, .
\end{align*}

These regularization terms are all Fr\'echet differentiable. Let $(\vz, r) \in \mathcal{M}$ and consider perturbations~$(\vu, \delta) \in \mathcal{M}$. The perturbation of the curve's parametrization $\vh$ is assumed to be defined by $n+1$ control points $\boldsymbol{Q} = (\vu^{0}, \ldots, \vu^{(n)})$. 

\begin{lemma}
Let $\vz$, $r$, $\vP$, $\vu$, $\delta$, $\boldsymbol{Q}$ be defined as above. There hold
\begin{multline*}
	\D \Psi_1 [\vz] \vu
	= \int_{0}^{1} \Bigg(
	2 \, \frac{\vz''(\tau) \cdot \vu''(\tau)}{|\vz'(\tau)|^3}
	- 3 \, |\vz''(\tau)|^2 \, \frac{\vz'(\tau) \cdot \vu'(\tau)}{|\vz'(\tau)|^5} \\
	- 2 \, \big( \vz'(\tau) \cdot \vu''(\tau) + \vz''(\tau) \cdot \vu'(\tau) \big) \, \frac{\vz'(\tau) \cdot \vz''(\tau)}{|\vz'(\tau)|^5} 
	+ 5 \vz'(\tau) \cdot \vu'(\tau) \, \frac{(\vz'(\tau) \cdot \vz''(\tau)) ^ 2}{|\vz'(\tau)|^7} \Bigg) \; \D \tau
	 \, ,
\end{multline*}
\begin{multline*}
	\D \Psi_2 [\vP] \boldsymbol{Q} =
	2 \, \sum_{j=0}^{n-1} \left[ \left( \frac{1}{n-1} \sum_{i = 0}^{n - 1} \left| \vx^{(i+1)} - \vx^{(i)} \right| \right) - \left| \vx^{(j+1)} - \vx^{(j)} \right| \right] \\
	\Bigg[ \left( \frac{1}{n - 1} \sum_{i=0}^{n-1} \frac{ (\vx^{(i+1)} - \vx^{(i)}) \cdot (\vu^{(i+1)} - \vu^{(i)})}{\left| \vx^{(i+1)} - \vx^{(i)} \right|} \right) 
	- \frac{(\vx^{(j+1)} - \vx^{(j)}) \cdot (\vu^{(j+1)}- \vu^{(j)}) }{\left| \vx^{(j+1)} - \vx^{(j)} \right|} \Bigg] ,
\end{multline*}
as well as
\begin{align*}
	\D \Psi_3 [r] \delta = 2 \int_0^1 r'(\tau) \, \delta'(\tau) \; \D \tau \, .
\end{align*}
Furthermore, we have
\begin{align*}
	\D \Psi_4 [r] \delta = \int_0^1 \D f_{r_{\min}}^l[r] \delta(\tau) \; \D \tau
\end{align*}
with
\begin{align*}
	\D f_{r_{\min}}^l[r(t)] \delta(t) =
	\begin{cases}
		0 \, , & \quad r(t) \leq r_{\min} - l \, , \\[1ex]
		- \frac{1}{2 \pi l^3} \, \delta(t) (r(t) + l - r_{\min}) \, \exp\left({-\frac{(r(t) + l - r_{\min})^2}{2 \, l^2}}\right) , & \quad r(t) > r_{\min} - l \, ,
	\end{cases}
\end{align*}
for $t \in [0, 1]$.
\end{lemma}

\begin{proof}
	For the Fr\'echet derivative of $\Psi_1$, see \cite[Lmm. 5.21]{Knoel2023} or \cite{WuLi2011, CapGri2021}.

	The Fr\'echet derivatives of the second and third term follow from
	\begin{align*}
		\D (| \cdot |) [\vx] \vu = \frac{\vx \cdot \vu}{| \vx |}
		\qquad \text{and} \qquad
		\D (( \,\cdot\, )^2) [x] u = 2 \, x u \, ,
	\end{align*}
	respectively.
	Finally, let $\delta$ be a sufficiently small perturbation of $r$. Then,
	\begin{align*}
		f_{r_{\min}}^l(&r(t) + \delta(t)) - f_{r_{\min}}^l(r(t)) - \D f_{r_{\min}}^l[r] \delta(t) \\
		&= \frac{1}{2 \pi l} \, \exp\left(-\frac{(r(t) + l - r_{\min})^2}{2 \, l^2} \right) \\
		&\hspace{4em}\Bigg[ \exp\left(-\frac{\delta^2(t) + 2 \delta(t) ( r(t) + l - r_{\min} )}{2 \, l^2} \right)
		- 1 + \frac{1}{l^2} \delta(t) ( r(t) + l - r_{\min} ) \Bigg] \\
		&= \frac{1}{2 \pi l} \, \exp\left(-\frac{(r(t) + l - r_{\min})^2}{2 \, l^2} \right) \\
		&\hspace{4em}\Bigg[ -\frac{\delta^2(t)}{2 l^2} - \frac{1}{l^2} \delta(t) ( r(t) + l - r_{\min} ) + \mathcal{O}(\delta^2(t))
		+ \frac{1}{l^2} \delta(t) ( r(t) + l - r_{\min} ) \Bigg] \\
		&= \mathcal{O}(\delta^2(t)) \, ,
	\end{align*}
	as $\Vert \delta \Vert_{1,\infty} \to 0$, for $t \in [0, 1]$. This finishes the proof.
\end{proof}

\section{Implementation and Numerical Examples}
\label{sec:NumericalExamples}

To implement our method for solving the optimization problem \eqref{eq:reg_optim_prob} as described in the previous section, we need to discretize the far field operator. We use appropriate linear combinations of the vector spherical harmonics defined in \eqref{eq:vec_sph_harmonics} to obtain complete orthonormal systems $B^\pm$ of the subspaces $V^\pm \subseteq L^2_t(\SS^2)$ defined  in \eqref{eq:V_pm} (see also \cite[Rem. 4.4]{FerGri2023}). Because of the properties of these subspaces, $B = B^+ \cup B^-$ is a complete orthonormal system in $L^2_t(\SS^2)$ and we can represent the far field operator by the projections
\[
  \left\langle \mathcal{F}_D \vphi, \vpsi \right\rangle_{L^2_t} \, , \qquad \vphi, \vpsi \in B \, .
\]
This representation is advantageous, as it also naturally yields representations of the component operators $\mathcal{F}_D^{pq}$, $p$, $q \in \{ +, - \}$, associated with the fields of specific helicities. For the numerical implementation, we restrict to spherical harmonics of degree at most $N \in \NN$, yielding a finite basis $B_N$ with $2N(N+2)$ elements.

Typically, we use B-spline curves with 20 control points in our optimizations. Hence, for the simplest case of a constant radius function along the length of the tube, there are 60 real degrees of freedom for our domains. Thus, even for a modest value of $N = 3$, we need to carry out $2N (N+2) \cdot 60 + 1 = 1801$ numerical solves of the Maxwell system in each iteration step. It cannot be stressed enough that the efficiency of the solver used for the EFIE is thus of utmost importance. Note however, that the boundary integral operator is the same for all these problems, they differ only in the boundary condition prescribed on $\partial D$. Thus, it makes sense to employ a direct solver for the corresponding linear systems, as the LU decomposition only needs to be computed once and the solutions of the individual problems are then easily obtained by forward and backward substitution. For the discretization of the boundary integral equations, the Python library \texttt{bempp-cl} (\cite{BetScr2021}, see also \url{https://bempp.com}) is employed, in which implementations of $\vS$, $\boldsymbol{\Psi}_{\mathrm{SL}}$ and $\boldsymbol{\Psi}_{\mathrm{SL}}^{\infty}$ from Section \ref{sec:OptDesignAlgorithm} are all available. Rao-Wilton-Glisson (RWG) basis functions are used to represent the solutions of the integral equations. However, the implementation of the boundary condition \eqref{eq:domain_derivative} requires additional work, as implementations of the surface differential operators are not readily available. In particular, a naive implementation leads to instabilities as pairing of RWG basis functions with scaled N\'ed\'elec basis functions violates the discrete inf-sup condition. A barycentric refinement of the mesh and corresponding Buffa-Christiansen basis functions for the dual basis are used to restore stability, as already described in the \red{prior} work \cite{AreKnoSch2026}. Moreover, all surface operators had to be reimplemented using just-in-time compilation techniques to keep the overall computational time at an acceptable level.

The boundary element mesh on the tube surface is generated from a grid defined on the parameter rectangle $(0, 1) \times (-\pi, \pi)$ of the parametrization \eqref{eq:tubelike_param} of the tube body. This step is carried out automatically in each step of the iteration. The degree of vector spherical harmonics needs to be chosen such that the far fields can be well approximated by these functions. The asymptotics discussed in \cite[Rem. 4.1]{GriSyl2018} show that we require at least $N > kR$ for an obstacle contained in a ball of radius $R$, where $k$ denotes the wave number.

Below, we give four examples of optimized tubular shapes, represented by $(\vz, r) \in \mathcal{M}$. In the first three examples we fix $r$ to the constant function $r \equiv 1.5$ while in the last example we carry out an optimization with varying $\vz$ and $r$. In addition to the update in each iteration step according to the BFGS scheme,  we further reduce the step length to satisfy the condition
\begin{align*}
	\| r \|_{\infty} \| \kappa \|_{\infty} < 1 \, ,
\end{align*}
where $\kappa$ denotes the curvature of the curve given by $\vz$. This condition prevents the object from having local self-intersections \cite[Thm.\ 1]{LitSimDur1999}. 
The optimization terminates if the relative update of the control points and (if applicable) the relative change of the radius function fall below a given tolerance.

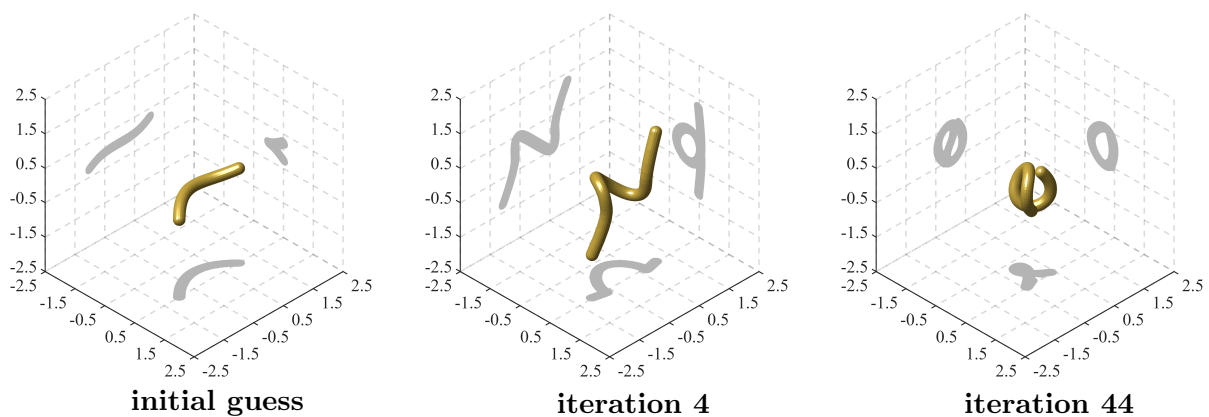
\begin{figure}[p]
	\begin{tabular}{ccc}
		\begin{tikzpicture}
			\node[anchor=south west,inner sep=0] (image) at (0,0) {\includegraphics[width=0.3\textwidth]{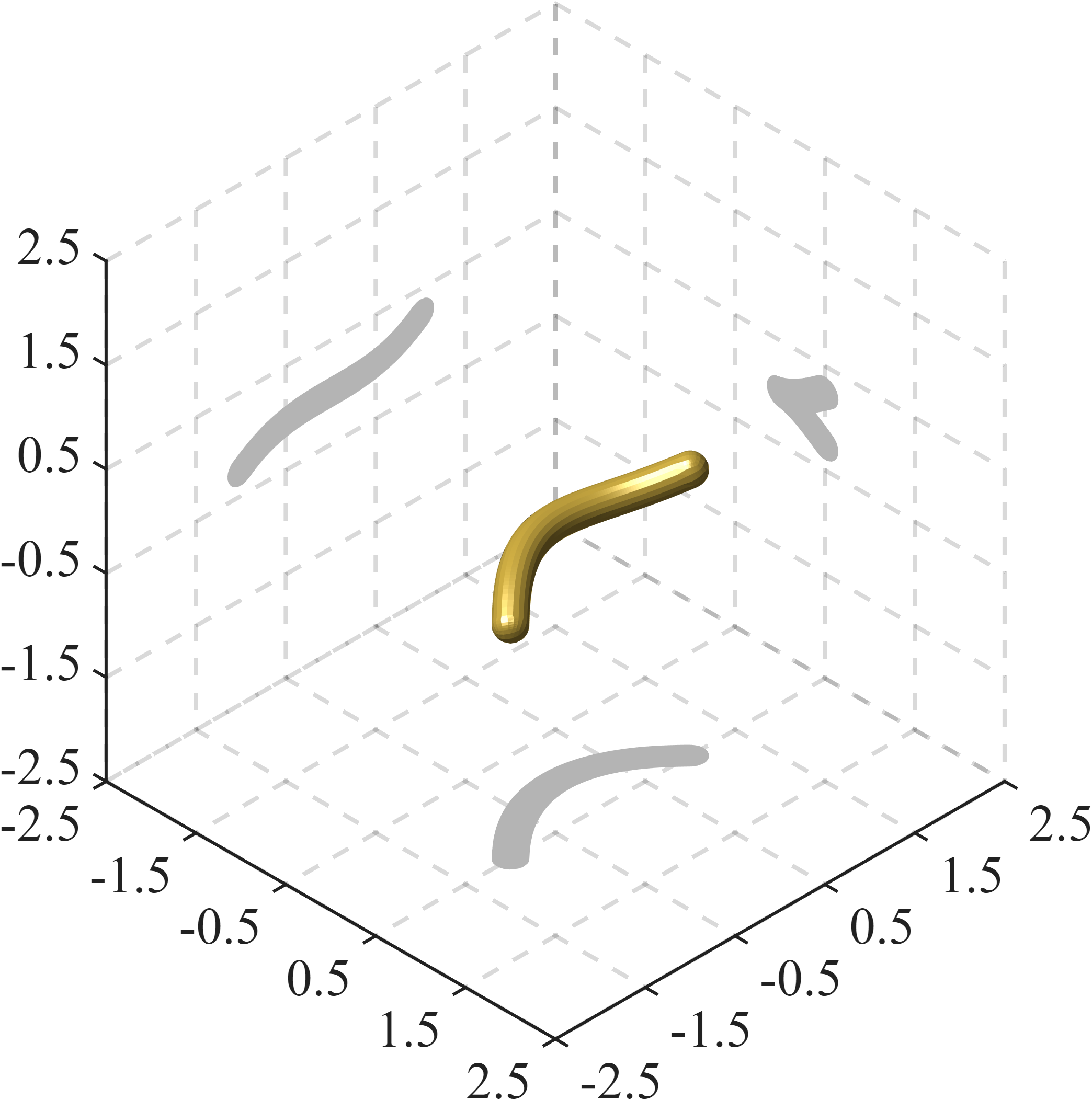}};
			\node[anchor=south west,inner sep=0] (image) at (5.5,0) {\includegraphics[width=0.3\textwidth]{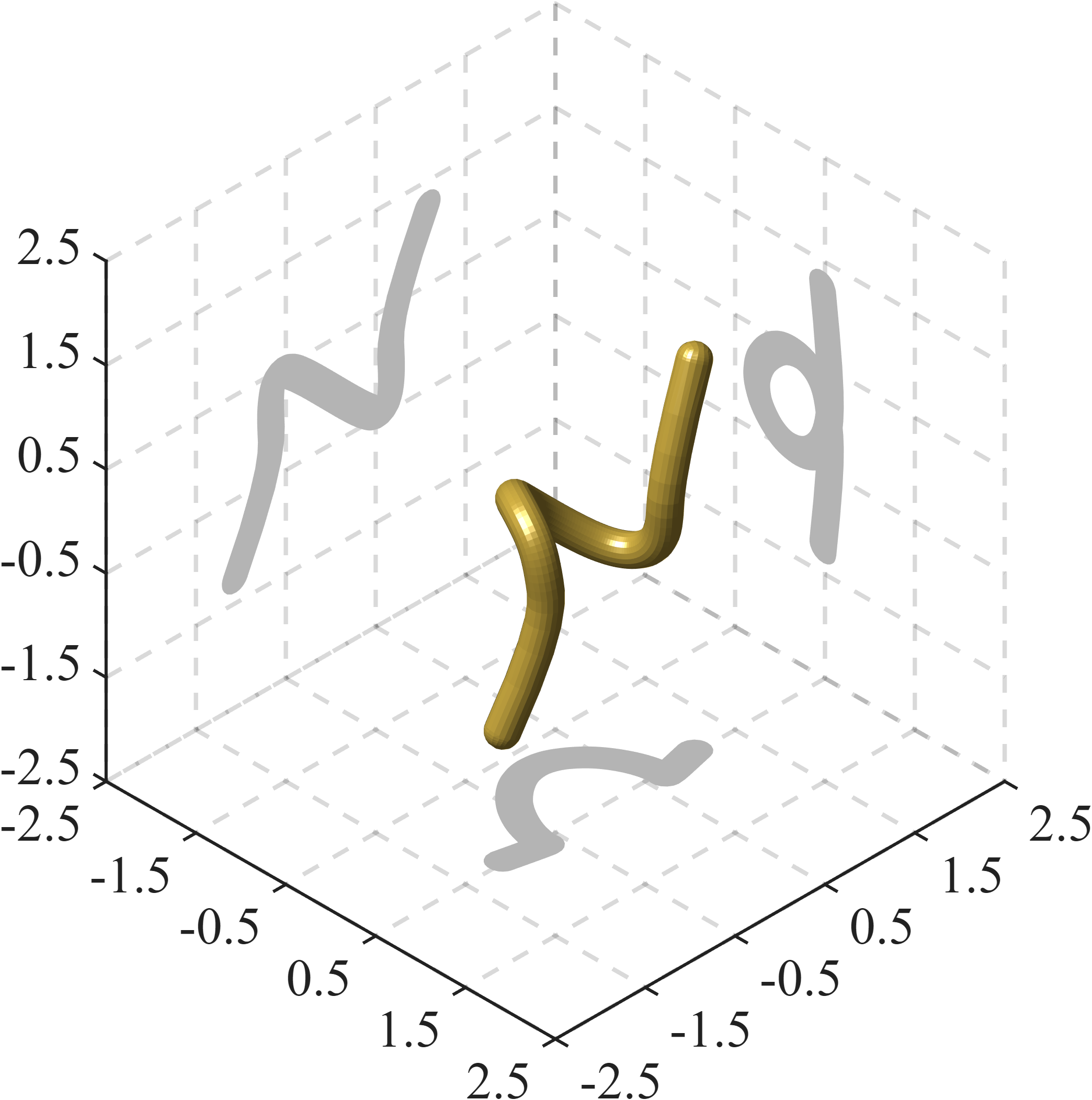}};
			\node[anchor=south west,inner sep=0] (image) at (11,0) {\includegraphics[width=0.3\textwidth]{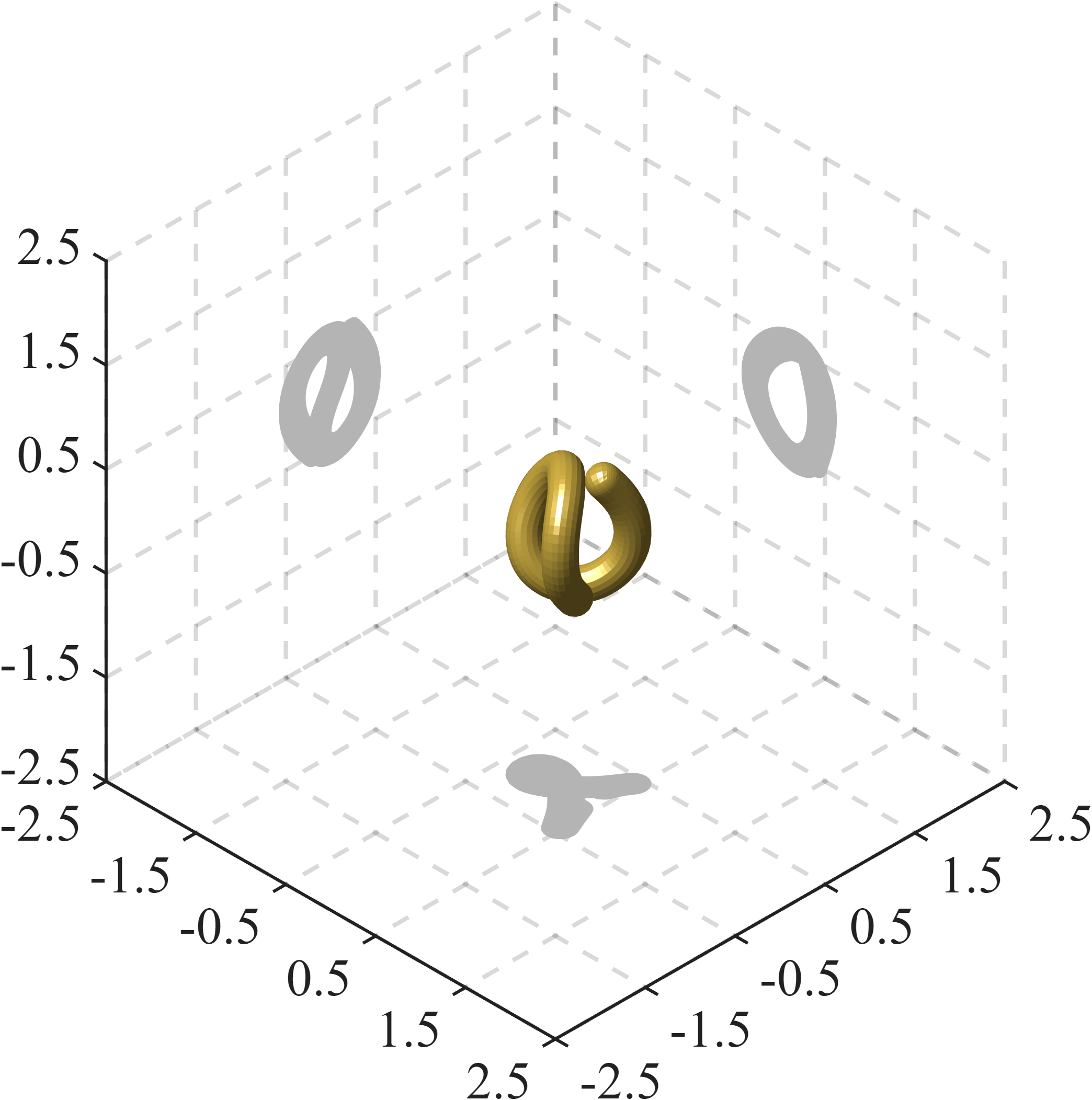}};
			\node at (2.75,-0.33) {\bf initial guess};
			\node at (8.25,-0.33) {\bf iteration 4};
			\node at (13.75,-0.33) {\bf iteration 44};
		\end{tikzpicture} 
	\end{tabular}
	\caption{initial guess, one intermediate and the final result for Example 1.}
	\label{fig:ex1}
\end{figure}

\begin{figure}[p]
	\begin{tabular}{ccc}
		\begin{tikzpicture}
			\node[anchor=south west,inner sep=0] (image) at (0,0) {\includegraphics[width=0.3\textwidth]{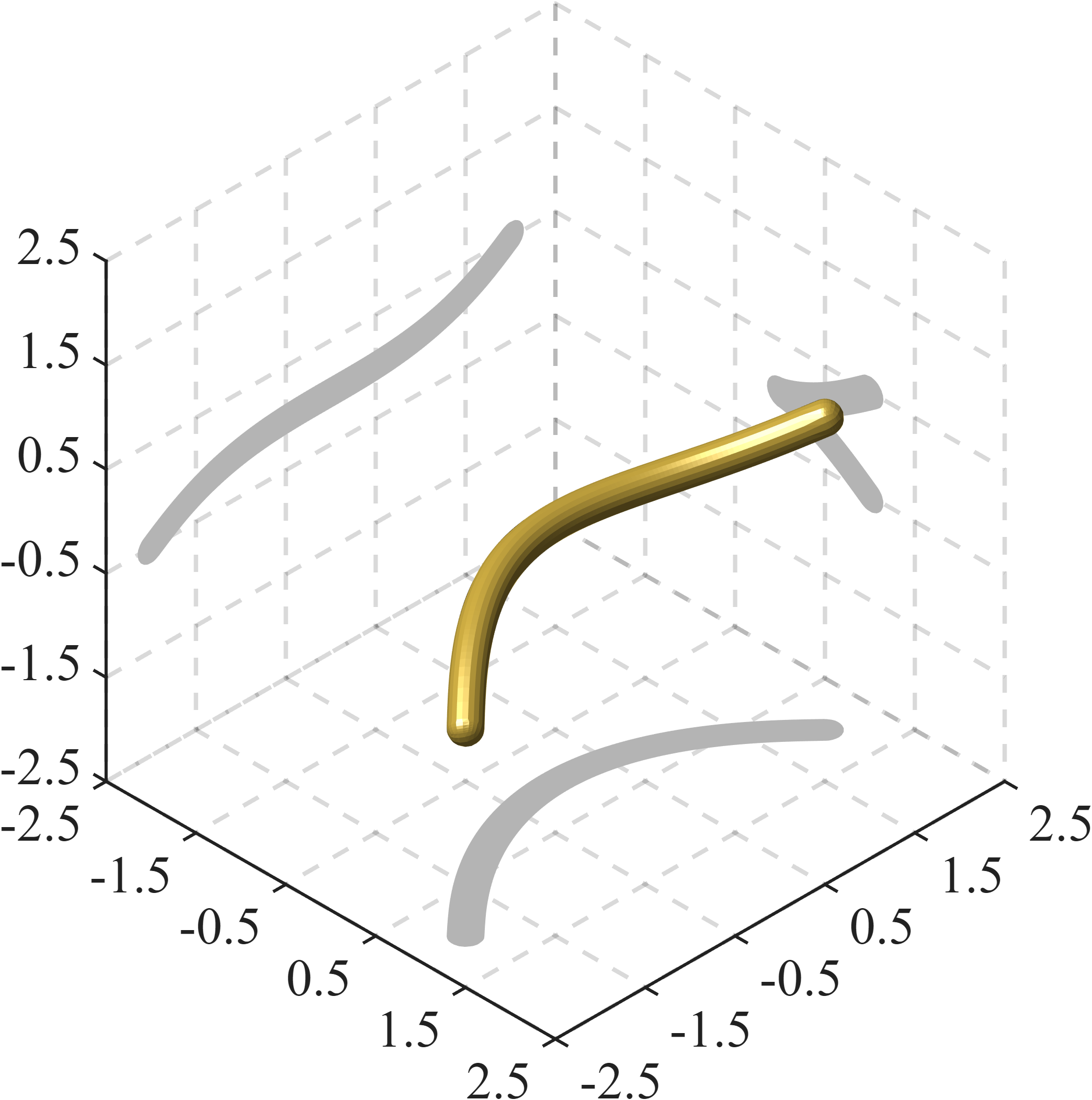}};
			\node[anchor=south west,inner sep=0] (image) at (5.5,0) {\includegraphics[width=0.3\textwidth]{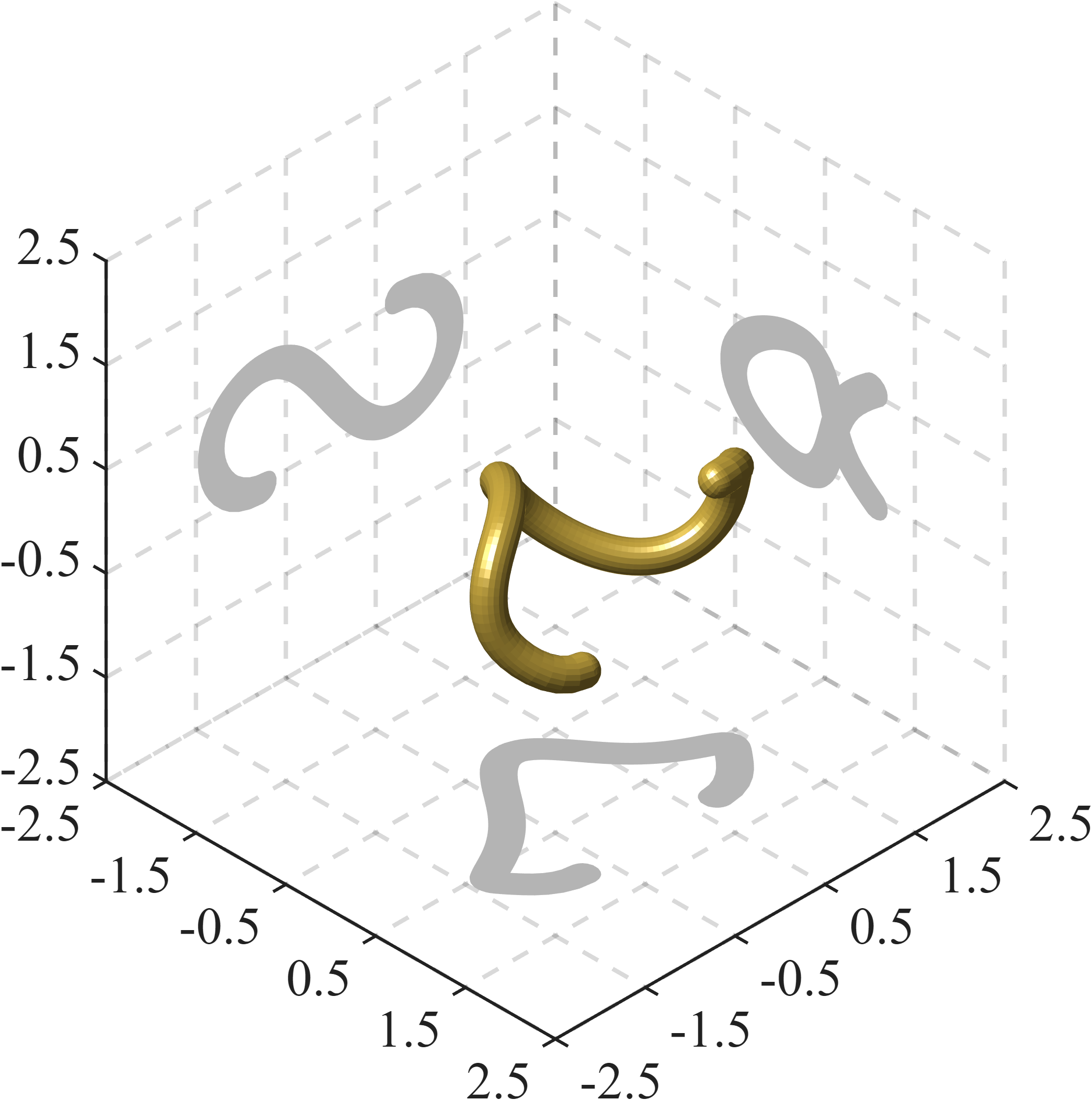}};
			\node[anchor=south west,inner sep=0] (image) at (11,0) {\includegraphics[width=0.3\textwidth]{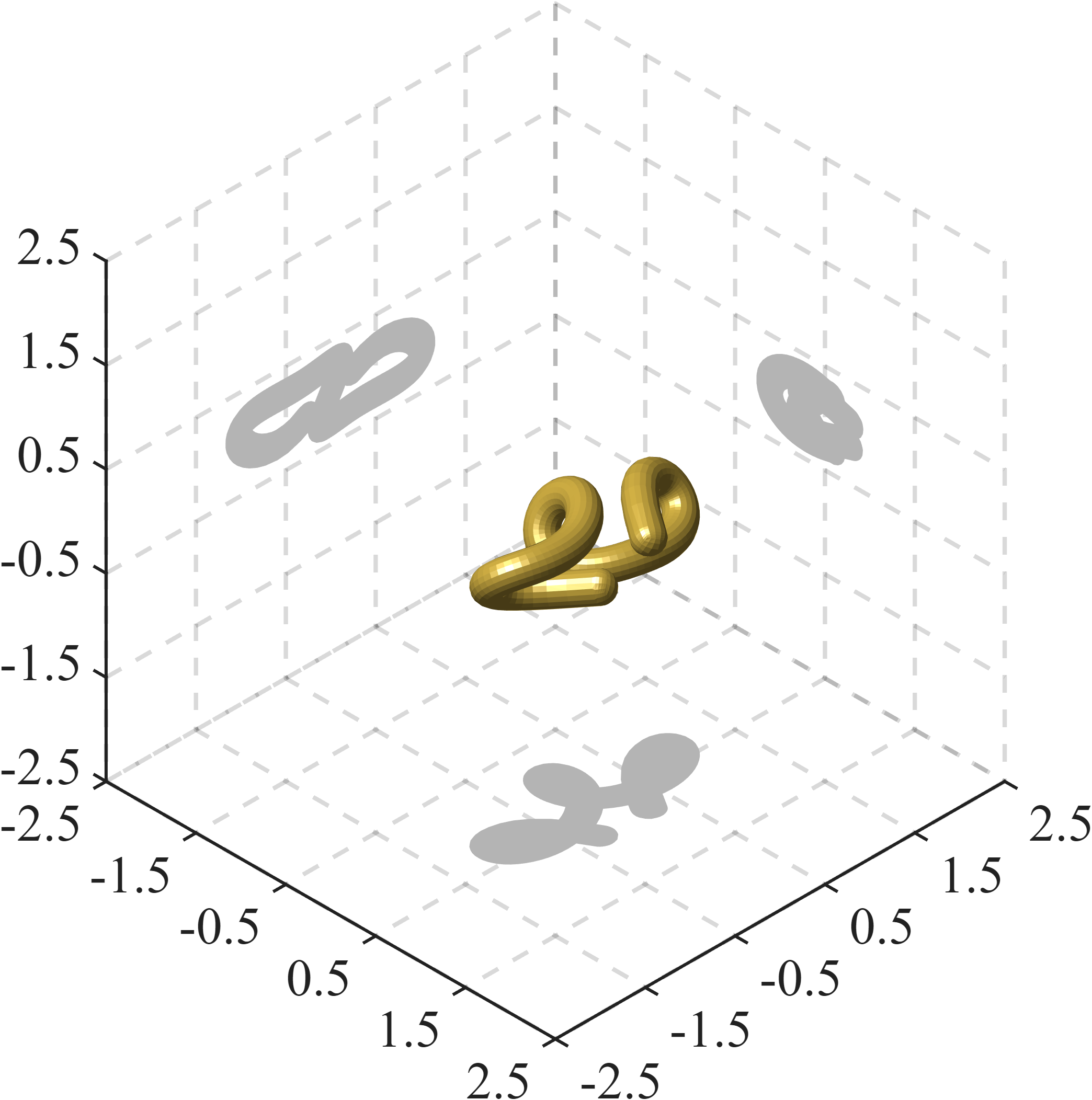}};
			\node at (2.75,-0.33) {\bf initial guess};
			\node at (8.25,-0.33) {\bf iteration 11};
			\node at (13.75,-0.33) {\bf iteration 31};
		\end{tikzpicture} 
	\end{tabular}
	\caption{initial guess, one intermediate and the final result for Example 2.}
	\label{fig:ex2}
\end{figure}

\begin{figure}[p]
	\begin{tabular}{ccc}
		\begin{tikzpicture}
			\node[anchor=south west,inner sep=0] (image) at (0,0) {\includegraphics[width=0.3\textwidth]{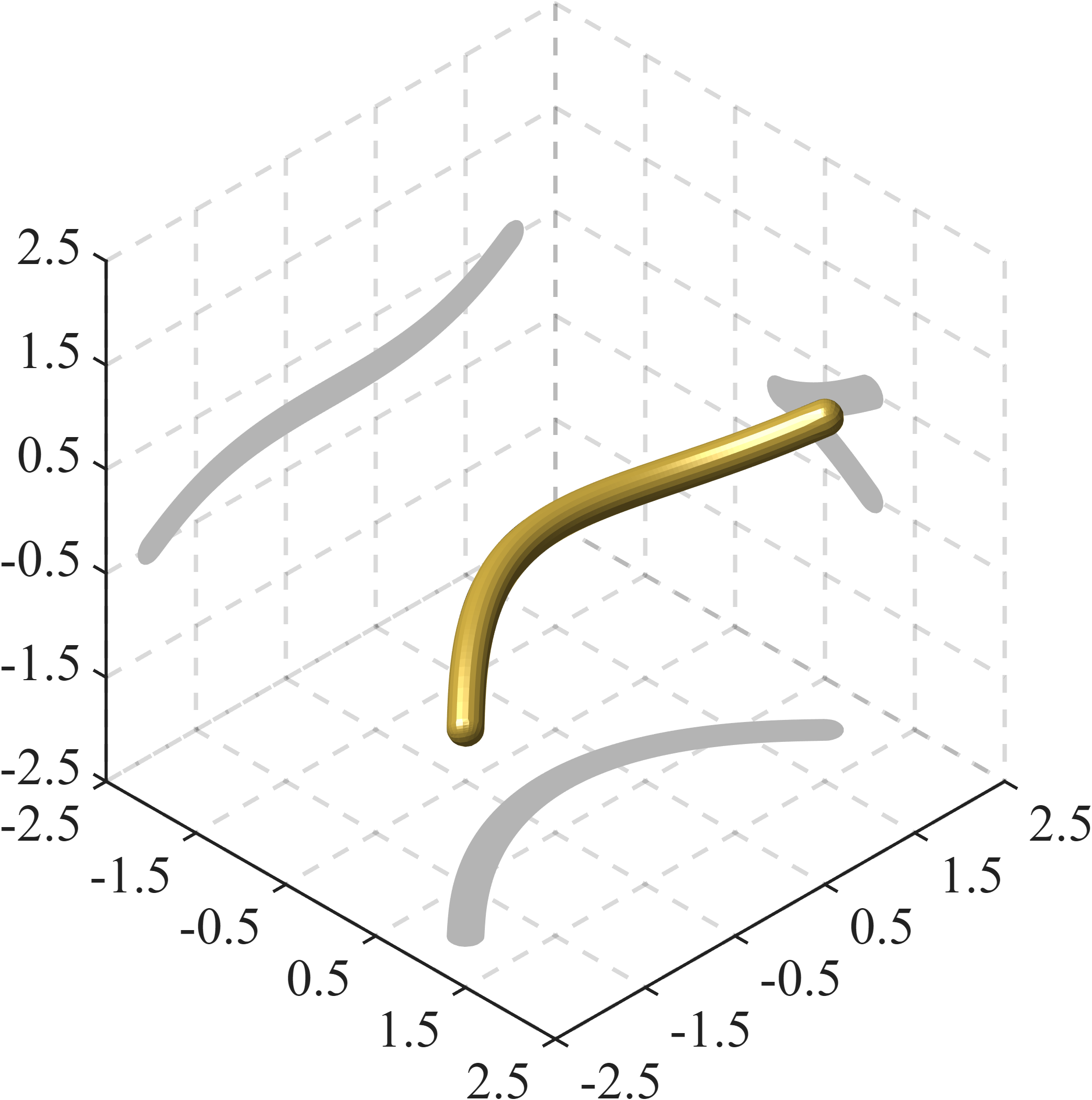}};
			\node[anchor=south west,inner sep=0] (image) at (5.5,0) {\includegraphics[width=0.3\textwidth]{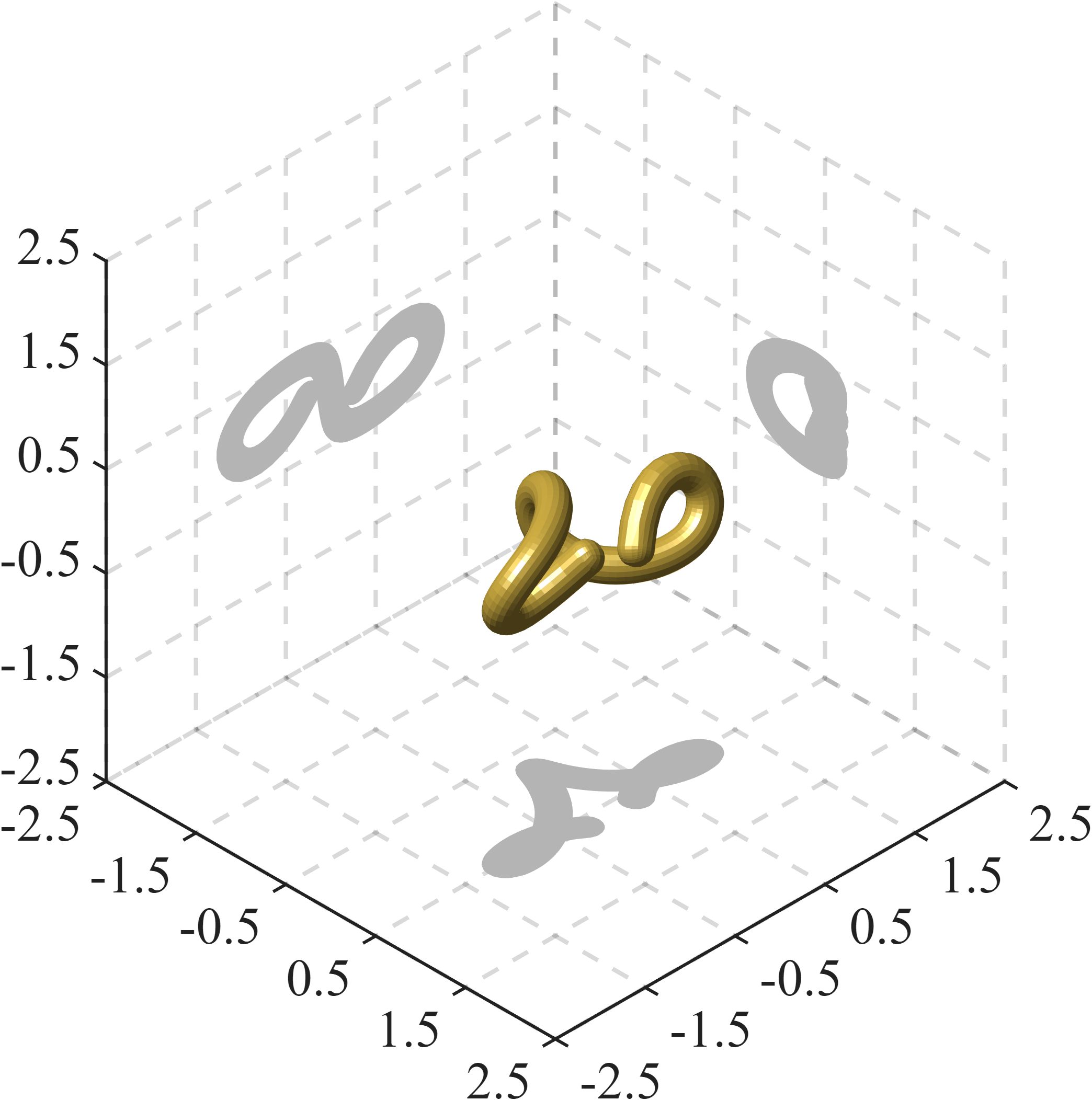}};
			\node[anchor=south west,inner sep=0] (image) at (11,0) {\includegraphics[width=0.3\textwidth]{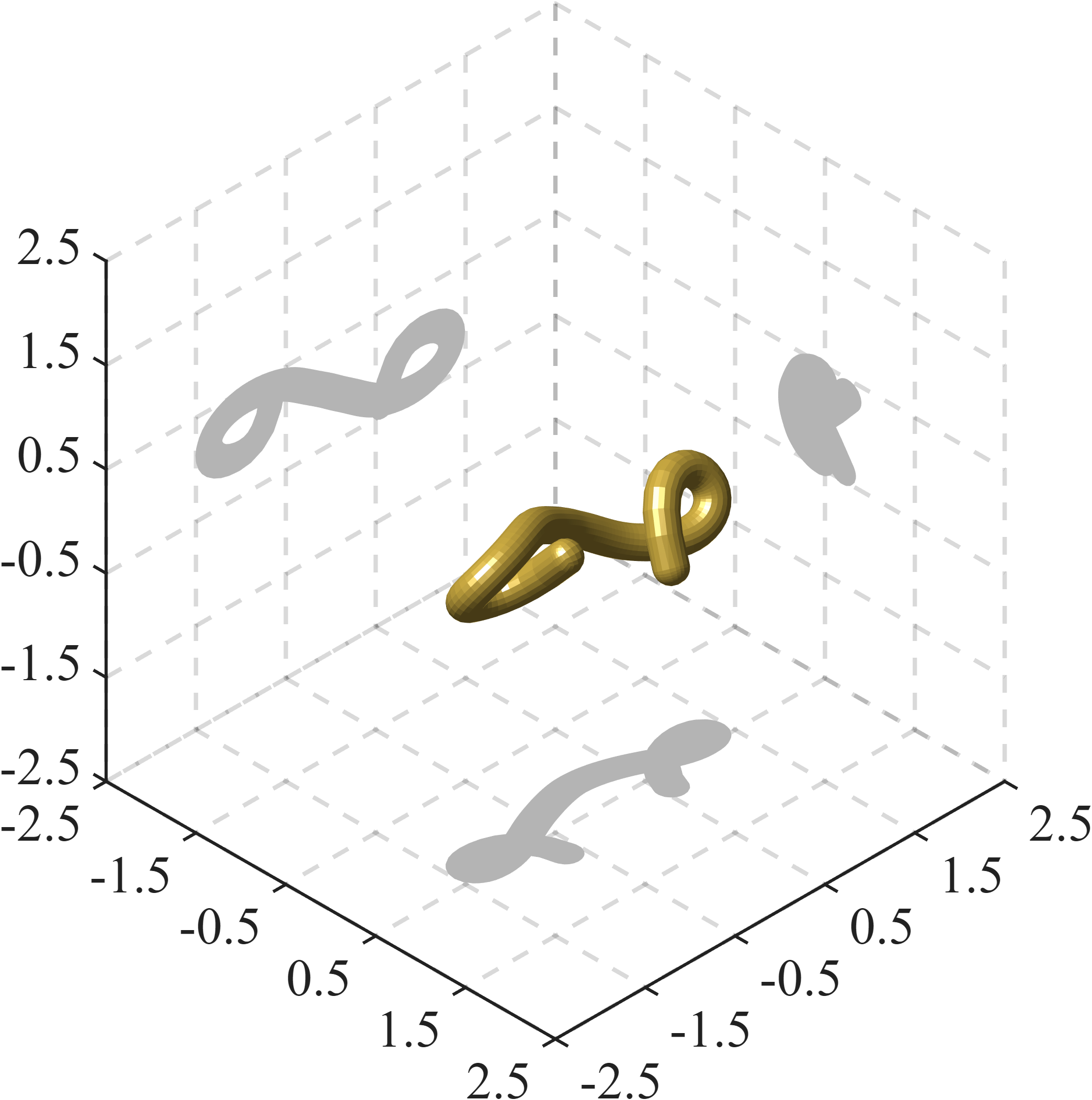}};
			\node at (2.75,-0.33) {\bf initial guess};
			\node at (8.25,-0.33) {\bf iteration 20};
			\node at (13.75,-0.33) {\bf iteration 44};
		\end{tikzpicture} 
	\end{tabular}
	\caption{initial guess, one intermediate and the final result for Example 3.}
	\label{fig:ex3}
\end{figure}
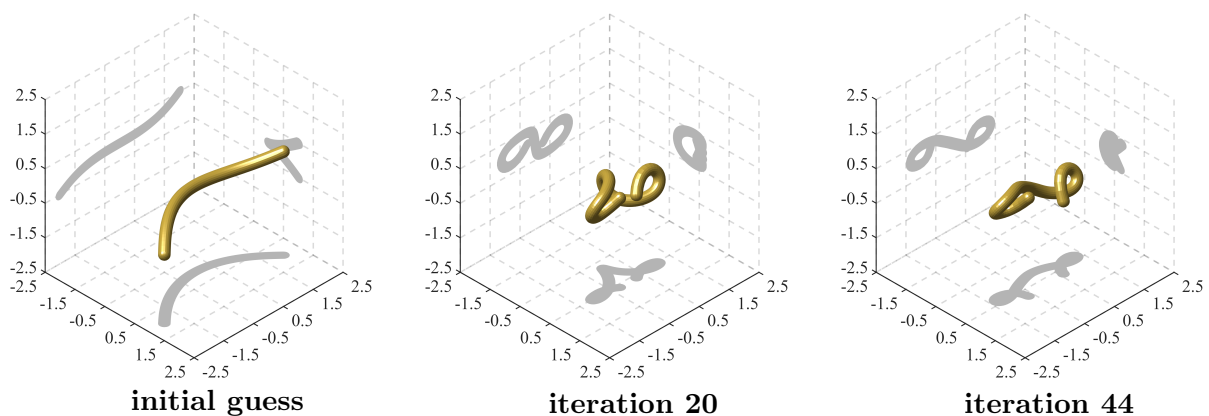


In all of the following examples, we use $\varepsilon_0 = 1$, $\mu_0 = 1$ and $\omega = 1$, so that the wavelength is $\lambda = 2\pi$ or, equivalently, $k = 1$. Note however, that by scaling corresponding objects may be obtained for any frequency where the perfect conductor model is valid. The parameters for the BFGS scheme and Armijo-type line search are chosen as in \cite{Knoel2023}, i.e., $\varepsilon_{\text{BFGS}} = 1\E{-5}$ and $\sigma = 1\E{-4}$, and we set $\delta = 0.5$.  As the algorithm would fail for an em-achiral initial configuration \red{because the gradient of the objective functional is not defined}, we always start the iteration with a slightly twisted tube.

\textbf{Example 1: } 
We initiate the optimization with an obstacle $D$ defined by a $3$rd-degree B-spline curve with control points $\boldsymbol{c}_j = 2 ( (j/20 - 1/2)^2, \, j/20 - 1/2, \, (j/20 - 1/2)^3 )^\top$, $j = 0, \ldots, 19$. The surface is discretized with a mesh and a corresponding RWG boundary element family with $6690$ degrees of freedom (DOFs).


The regularization parameters are chosen as $\alpha_1 = 5\E{-}3,\ \alpha_2 = 1\E{-}2$. As the radius function is assumed to be constant, the other regularization terms vanish. Due to the small size of the initial object, a value of~$N=2$ is sufficient to represent the far field operator. Thus 481 discretized boundary integral equations are solved in each iteration step. Starting with a relative em-chirality of approximately $2.67\%$ the BFGS scheme terminates after $44$ iterations yielding an object with a relative em-chirality of around $89.04\%$.
The results for the optimization are shown in Figure \ref{fig:ex1}.

\textbf{Example 2: } 
We initiate the optimization with an obstacle $D$ defined by a $3$rd-degree B-spline curve with control points $\boldsymbol{c}_j = 4 ( (j/20 - 1/2)^2, \, j/20 - 1/2, \, (j/20 - 1/2)^3 )^\top$, $j = 0, \ldots, 19$. The mesh again is chosen such that the RWG boundary element space has got dimension $6690$.
The regularization parameters are chosen as $\alpha_1 = 2.5\E{-}3,\ \alpha_2 = 1\E{-}2$. Due to the increased size of the initial object, we set $N=3$ and hence have to solve 1801 boundary integral equations in each iteration step. Starting with a relative em-chirality of approximately $2.8\%$ the BFGS scheme terminates after $31$ iterations yielding an object with a relative em-chirality of around $93.61\%$.
The results for the optimization are shown in Figure \ref{fig:ex2}.

\textbf{Example 3: } 
We initiate the optimization with same obstacle as in Example 2 and we also choose the same initial mesh and boundary element space as well as $N=3$. However, the regularization parameters are chosen as $\alpha_1 = 3\E{-}3,\ \alpha_2 = 1\E{-}2$.
Starting with a relative em-chirality of approximately $2.8\%$ the BFGS scheme terminates after $44$ iterations yielding an object with a relative em-chirality of around $94.11\%$.
The results for the optimization are shown in Figure \ref{fig:ex3}.

In all three examples we reliably achieve an optimized shape with a very high em-chirality measure. The size of the object may be influenced by appropriate choices of the initial object as well as the regularization parameter $\alpha_1$ that controls the overall curvature of the spine curve. As the following example shows, we were less successful in also optimizing the radius function along the length of the curve. Either, as in the example presented here, the radius becomes very small along sections of the tube or, by choosing a sufficiently large regularization parameter $\alpha_3$ and $\alpha_4$, this effect is prevented at the cost of stagnation of the optimization at rather low measures of em-chirality.

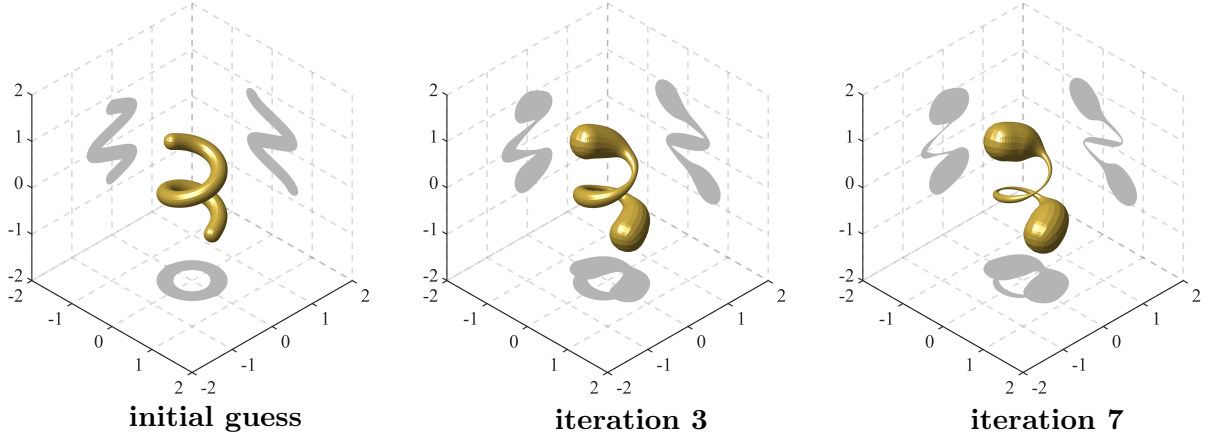
\begin{figure}[t]
	\begin{tabular}{ccc}
		\begin{tikzpicture}
			\node[anchor=south west,inner sep=0] (image) at (0,0) {\includegraphics[width=0.3\textwidth]{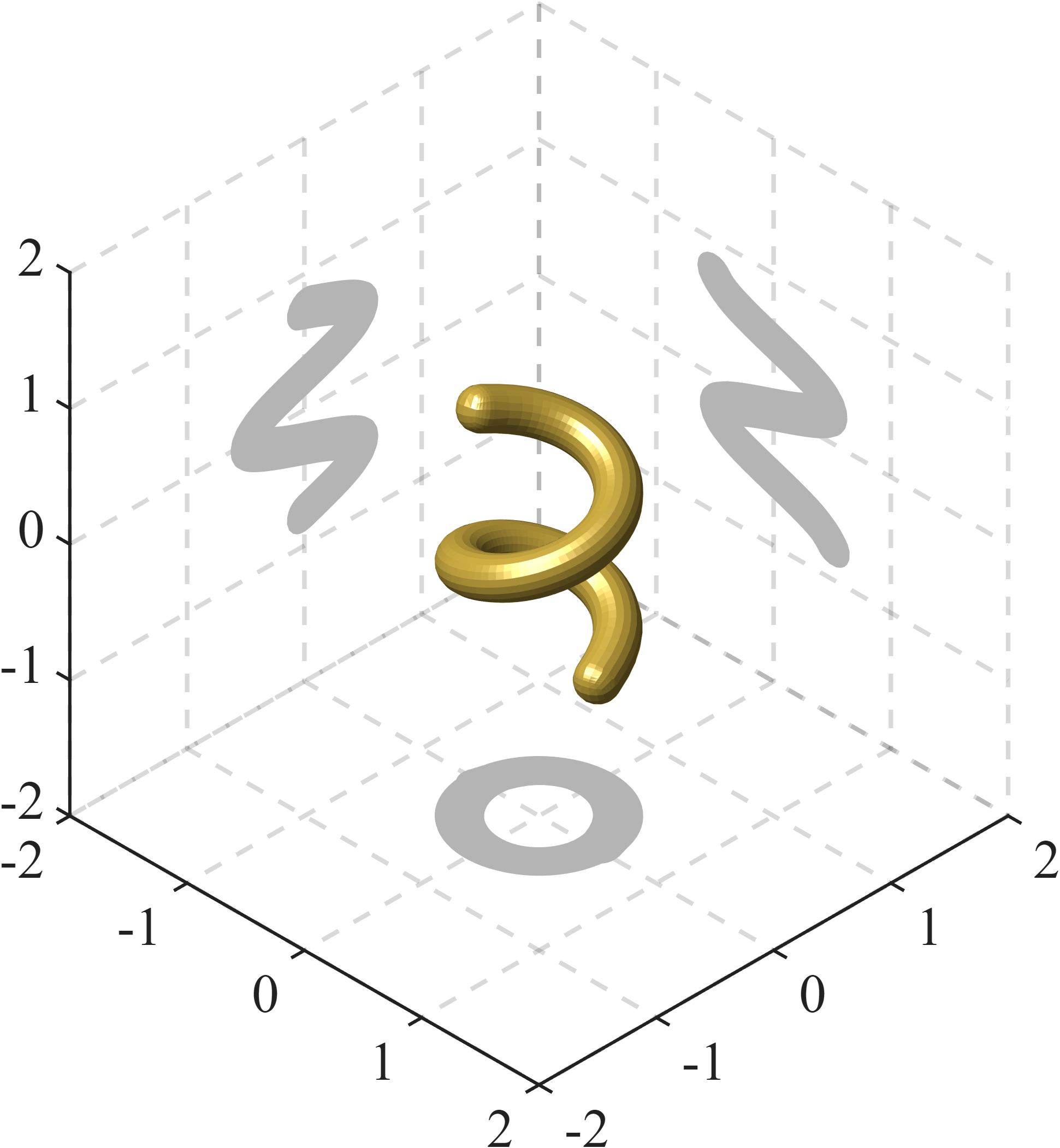}};
			\node[anchor=south west,inner sep=0] (image) at (5.5,0) {\includegraphics[width=0.3\textwidth]{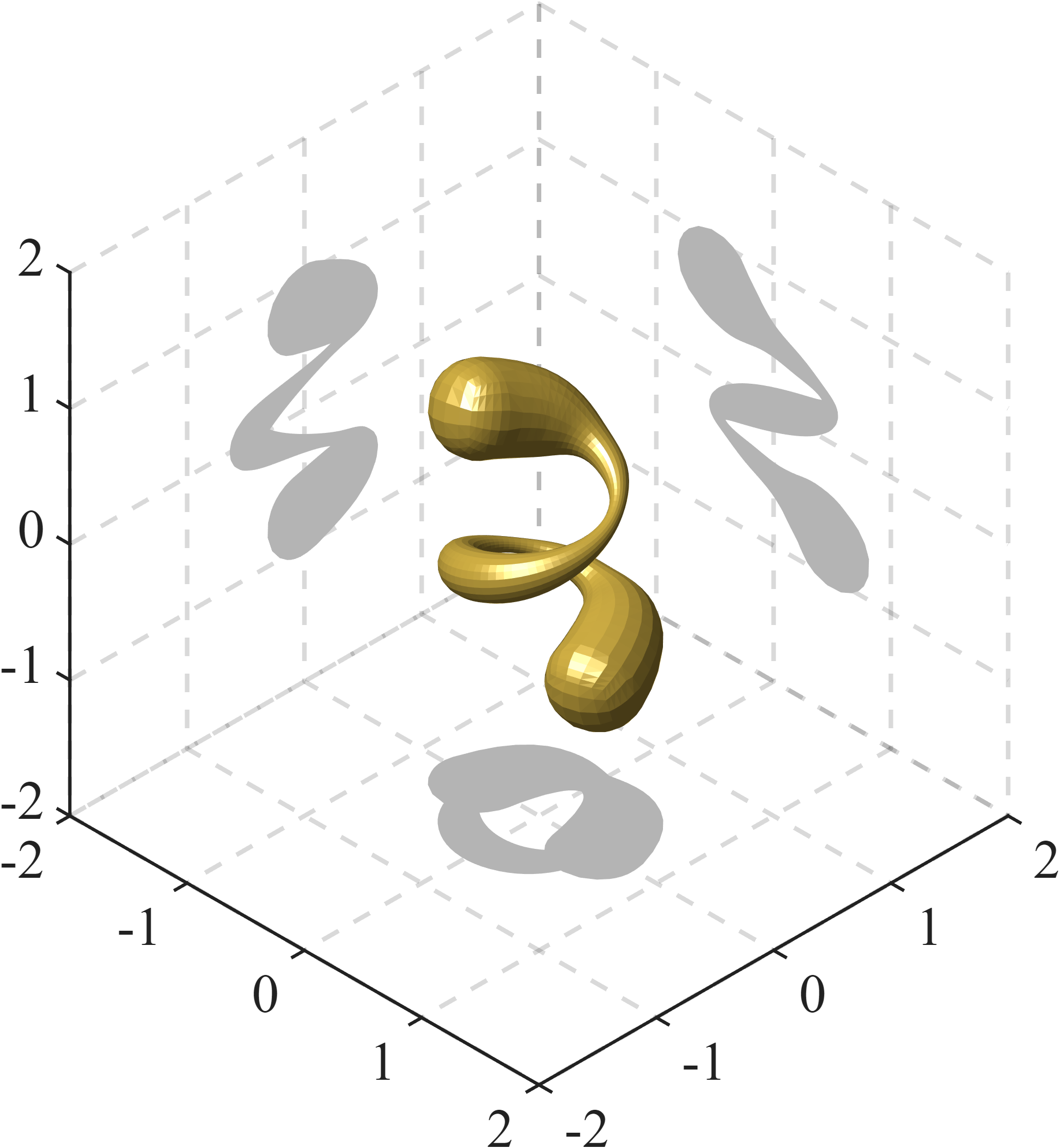}};
			\node[anchor=south west,inner sep=0] (image) at (11,0) {\includegraphics[width=0.3\textwidth]{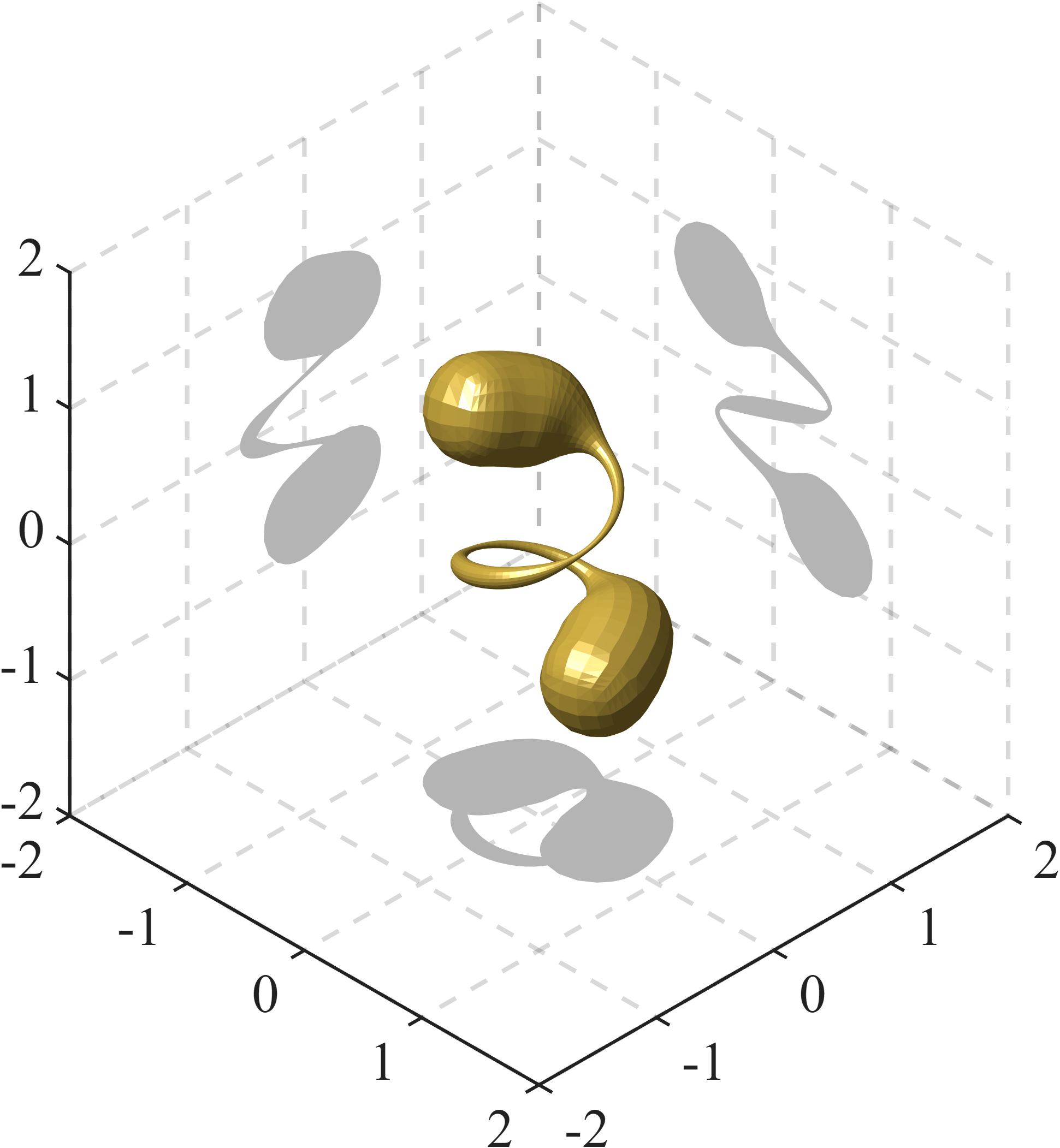}};
			\node at (2.75,-0.33) {\bf initial guess};
			\node at (8.25,-0.33) {\bf iteration 3};
			\node at (13.75,-0.33) {\bf iteration 7};
		\end{tikzpicture}%
	\end{tabular}
	\caption{initial guess, one intermediate and the final result for Example 4.}
	\label{fig:ex4}
\end{figure}

\begin{figure}[p]
    \begin{center}
     \includegraphics[width=0.32\textwidth]{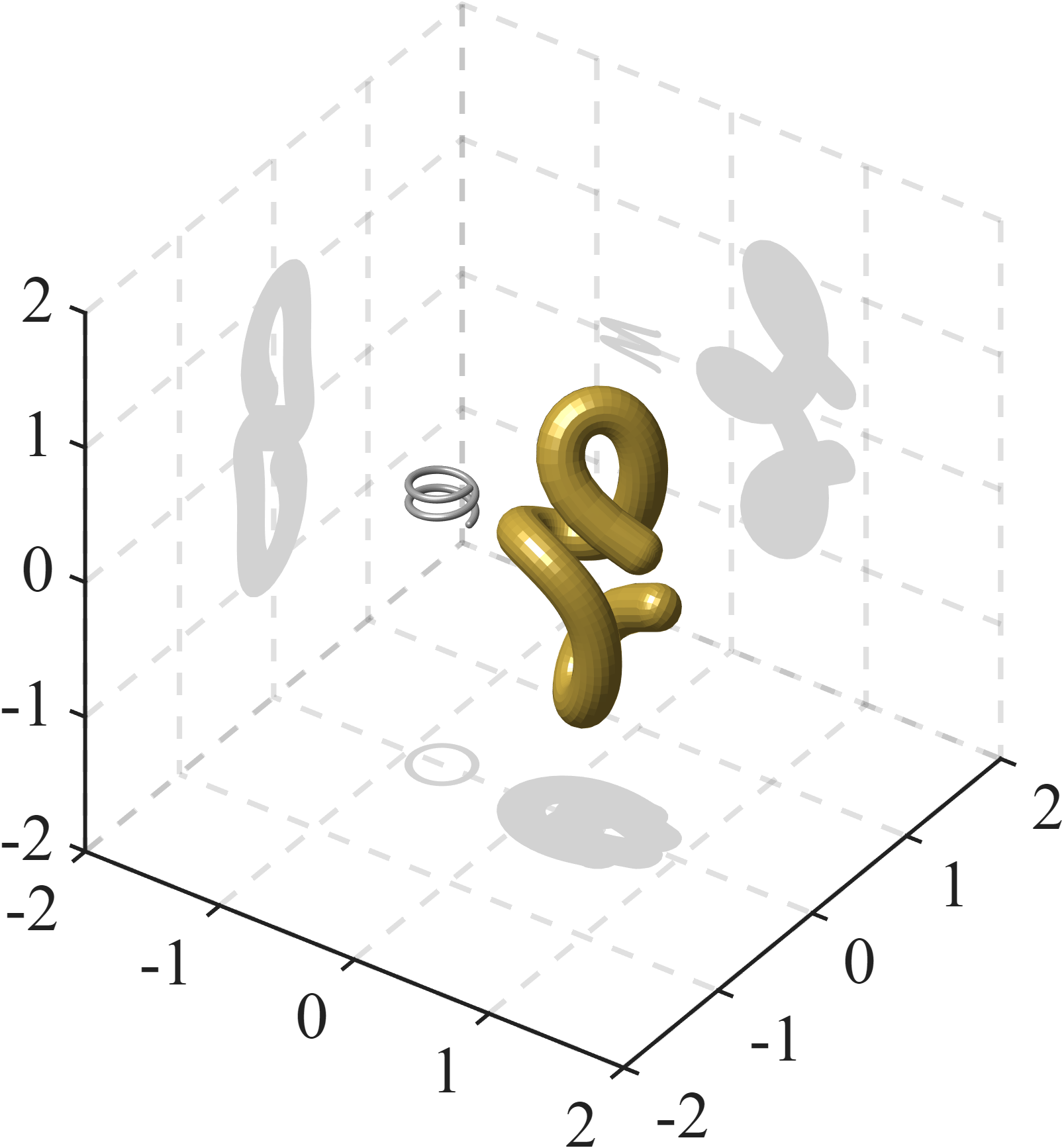}
			\hspace{4em}
			\begin{tikzpicture}[scale=0.8]
					\node[anchor=south west,inner sep=0] (image) at (0,0) {\includegraphics[trim={0 29 0 178},clip,width=0.4\textwidth]{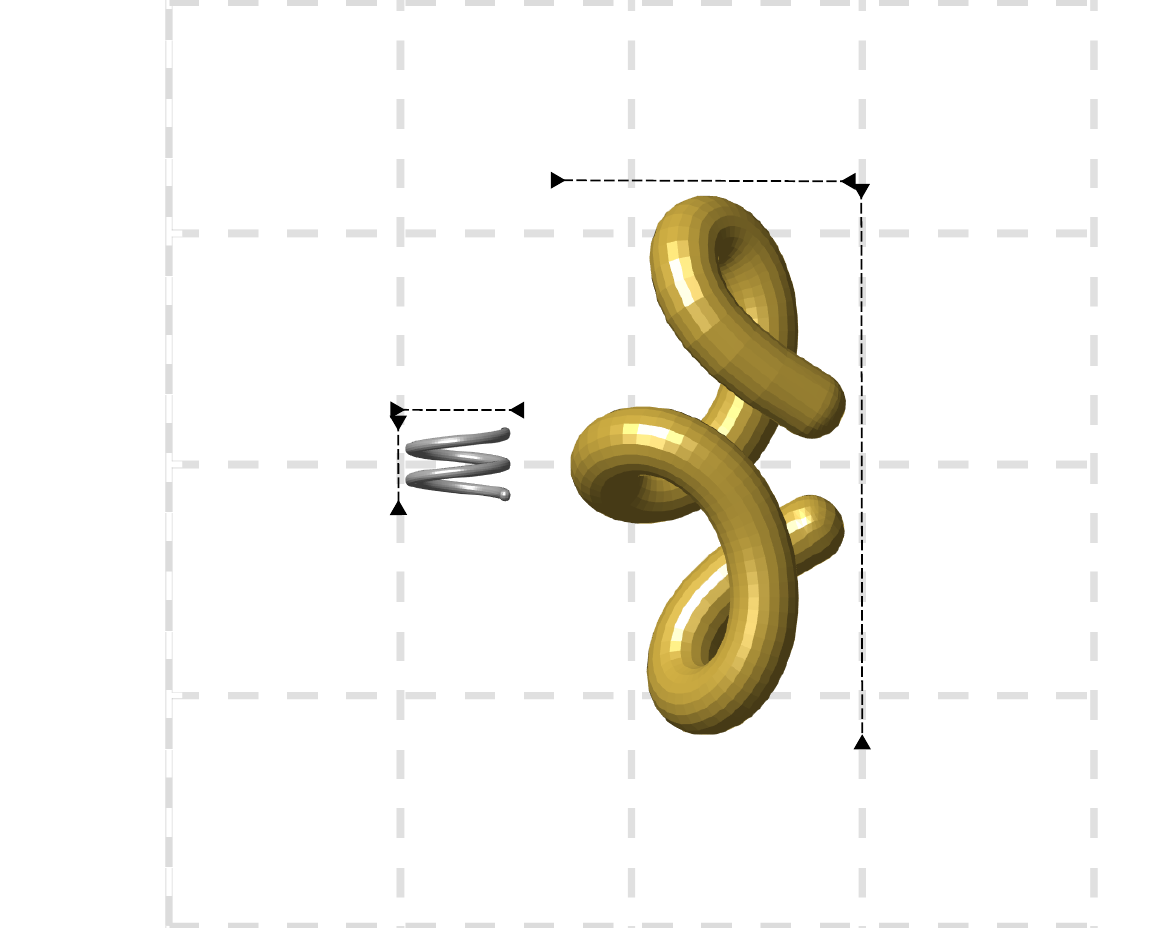}};
					\node at (3.3,  4.05) {\small $\SI{14.5}{\micro\metre}$};
					\node at (5.1,  5.7) {\small $\SI{37.8}{\micro\metre}$};
					\node at (2.05, 3.35) {\small $\SI{10.1}{\micro\metre}$};
					\node at (7.1,  3.35) {\small $\SI{74.2}{\micro\metre}$};
				\end{tikzpicture}
    \end{center}
	\caption{Size comparison of the optimized object from Example 2 and the silver helix from \cite{FerFruRoc2016}.
			The optimized object has a wire radius that is roughly six times larger, an overall height that is 10 times larger and a width that is 2.6 times larger.}
	\label{fig:comparison}
\end{figure}

\textbf{Example 4: } 
	In Figure \ref{fig:ex4}, we show an optimization result with respect to the radius function only.
	We start with a helix (one and a half turns) of radius $r \equiv 0.15$ that has a relative em-chirality of around $2.24\%$ and set $N=2$, $\alpha_3 = 1\E{-}3$, $\alpha_4 = 1\E{-}2$.
	For the fourth penalty term, we set $r_{\text{min}} = r / 10$ and $l = r_{\text{min}} / 20$.
	During the optimization, caps get bigger and the body gets thinner and after 7 iterations, we arrive at a relative em-chirality of around $73.59\%$.

We would like to discuss the feasibility of our designs by comparing to a typical object used in the literature as a highly em-chiral scatterer. The publication \cite{FerFruRoc2016} contains an example of a silver helix whose pitch, radii and number of revolutions have been optimized for high em-chirality at a frequency of $\lambda = \SI{200}{\micro\metre}$. At this wavelength, the perfect conductor is a good approximation for a silver object which allows us to compare our designs, appropriately scaled, to this helix \cite{JohChr1972}. Figure \ref{fig:comparison} provides a comparison of the actual sizes of the helix and our design from Example 2, showing that our object is substantially larger and constructed using a much thicker ``wire''. Figure \ref{fig:freq_scans} shows the Hilbert-Schmidt norm of the far field operators for both objects as well as two relative chirality measures for a range of frequencies around the value of $\SI{200}{\micro\metre}$. The chirality measure $\chi(\mathcal{F}_D)$ has been introduced in \cite{FerFruRoc2016} and always satisfies the estimates $\| \mathcal{F}_D \|_{\hs} \geq \chi(\mathcal{F}_D) \geq \chi_{\hs}(\mathcal{F}_D)$, but is less smooth than $\chi_{\hs}$. The Hilbert-Schmidt norm of $\mathcal{F}_D$, whose square is also termed \emph{total interaction cross section} in \cite{FerFruRoc2016}, is a measure of the intensity of scattered fields generated by the obstacle over all directions of incidence.  The displayed quantities for the silver helix are those of a resimulation of the computations from \cite{FerFruRoc2016} using boundary elements, and are in good agreement of the plots shown in that reference. The comparison shows that our design, like the substantially smaller silver helix, has a peak in both of the chirality measures and in $\| \mathcal{F}_D \|_\hs$ at $\lambda = \SI{200}{\micro\metre}$.

\begin{figure}[p]
	\begin{tabular}{cc}
		\includegraphics[trim={0 0 0 0},clip,width=0.47\textwidth]{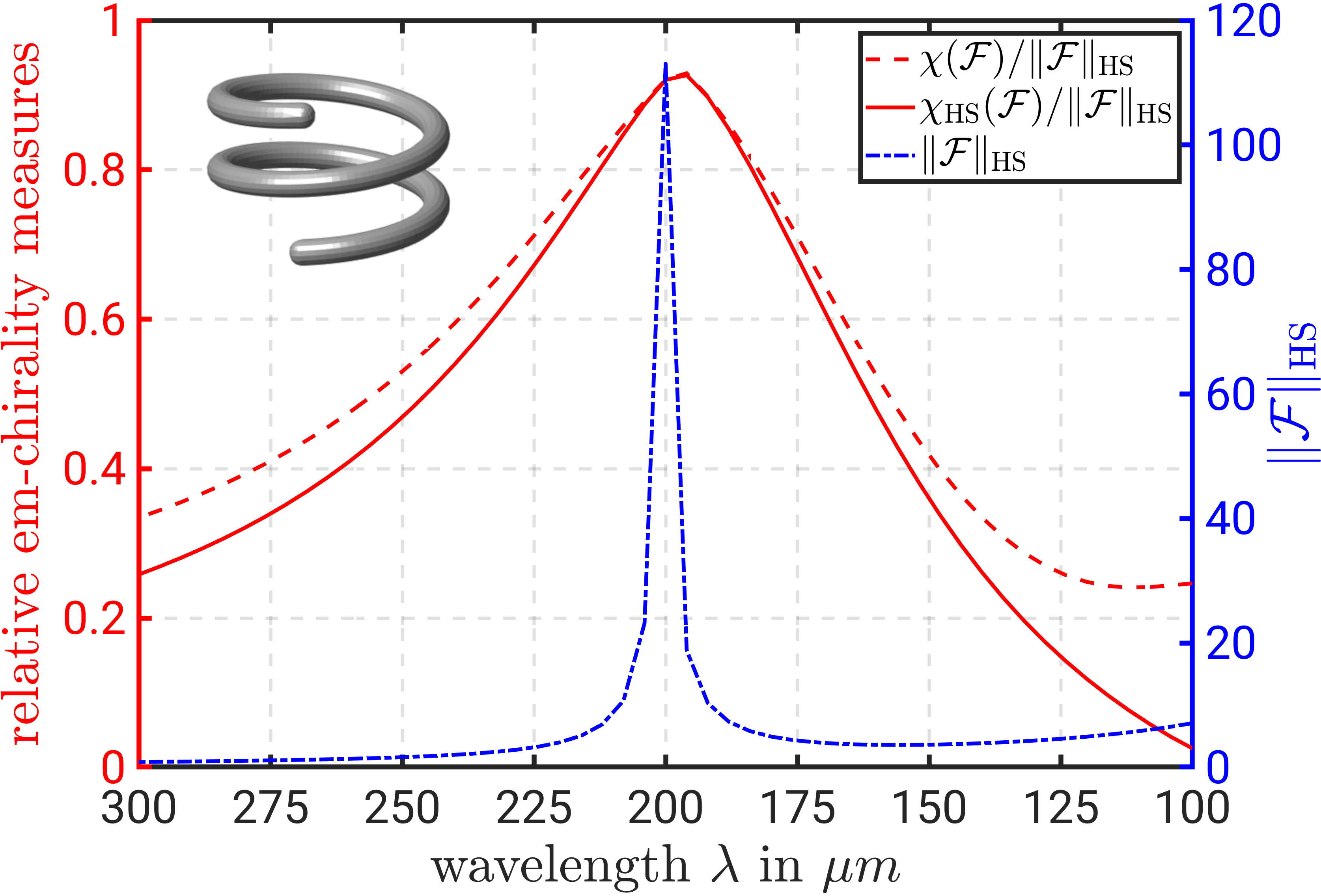}
		&
        \includegraphics[trim={0 0 0 0},clip,width=0.47\textwidth]{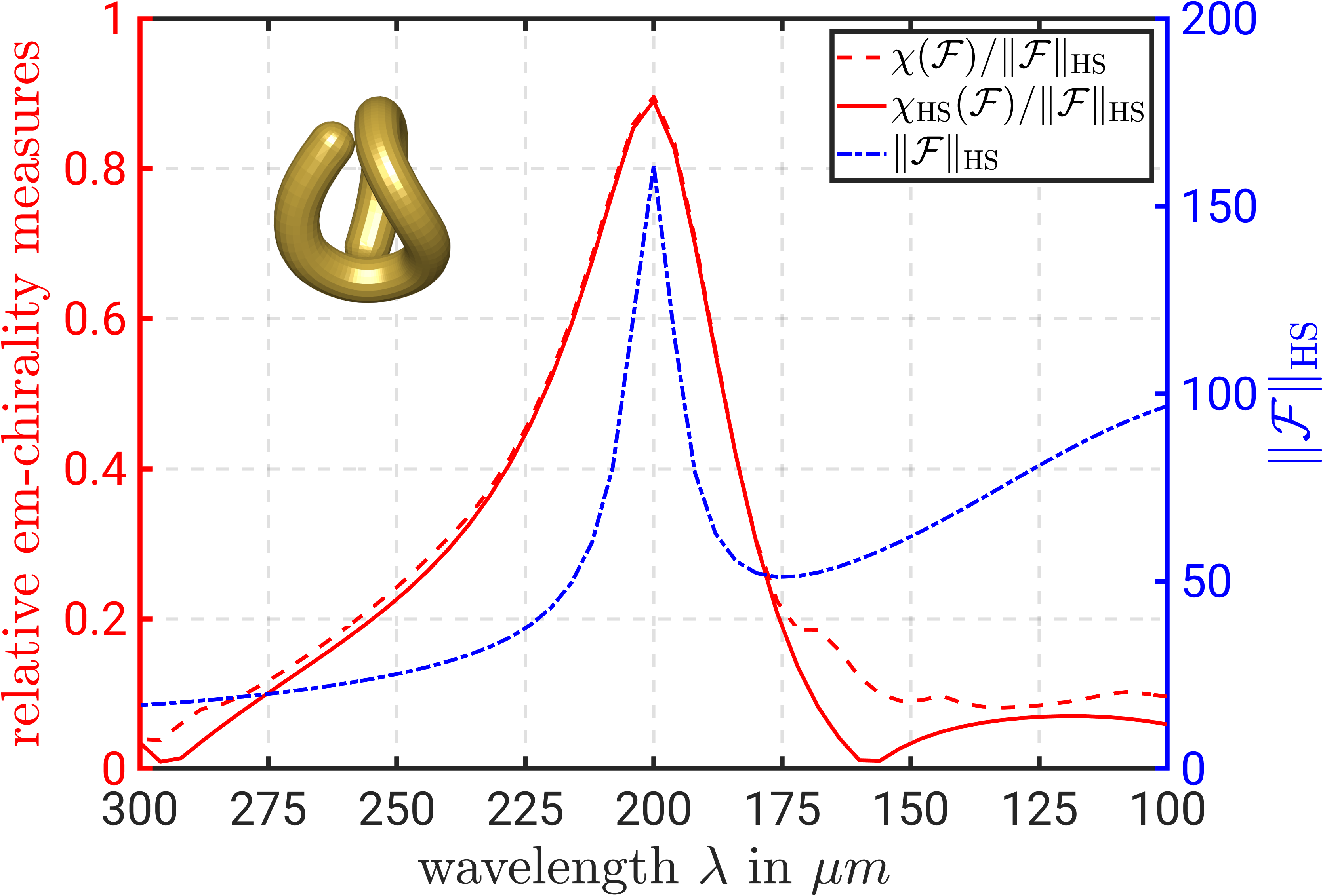} \\
        (a) & (b) \\[2ex]
        \includegraphics[trim={0 0 0 0},clip,width=0.47\textwidth]{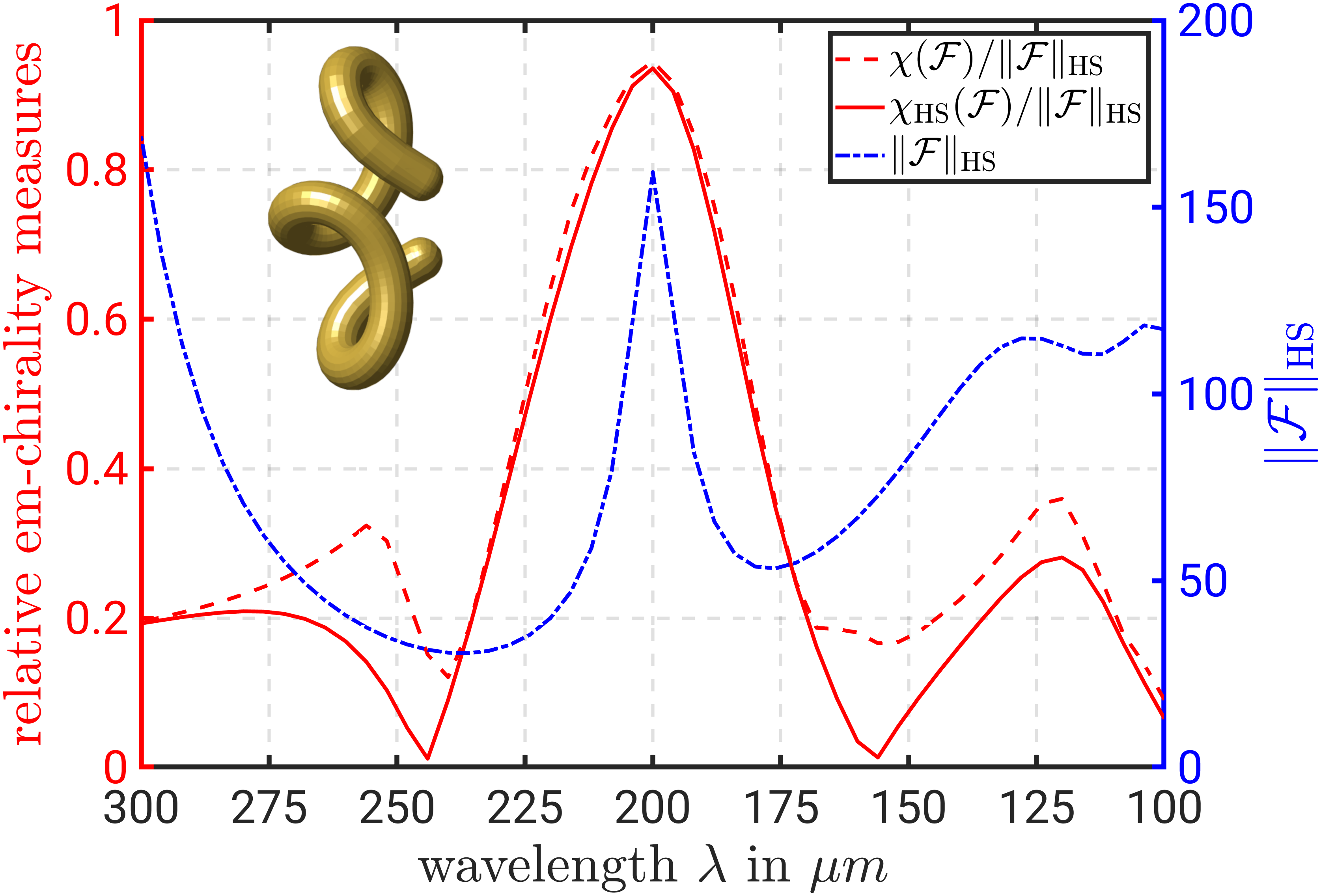}
        &
        \includegraphics[trim={0 0 0 0},clip,width=0.47\textwidth]{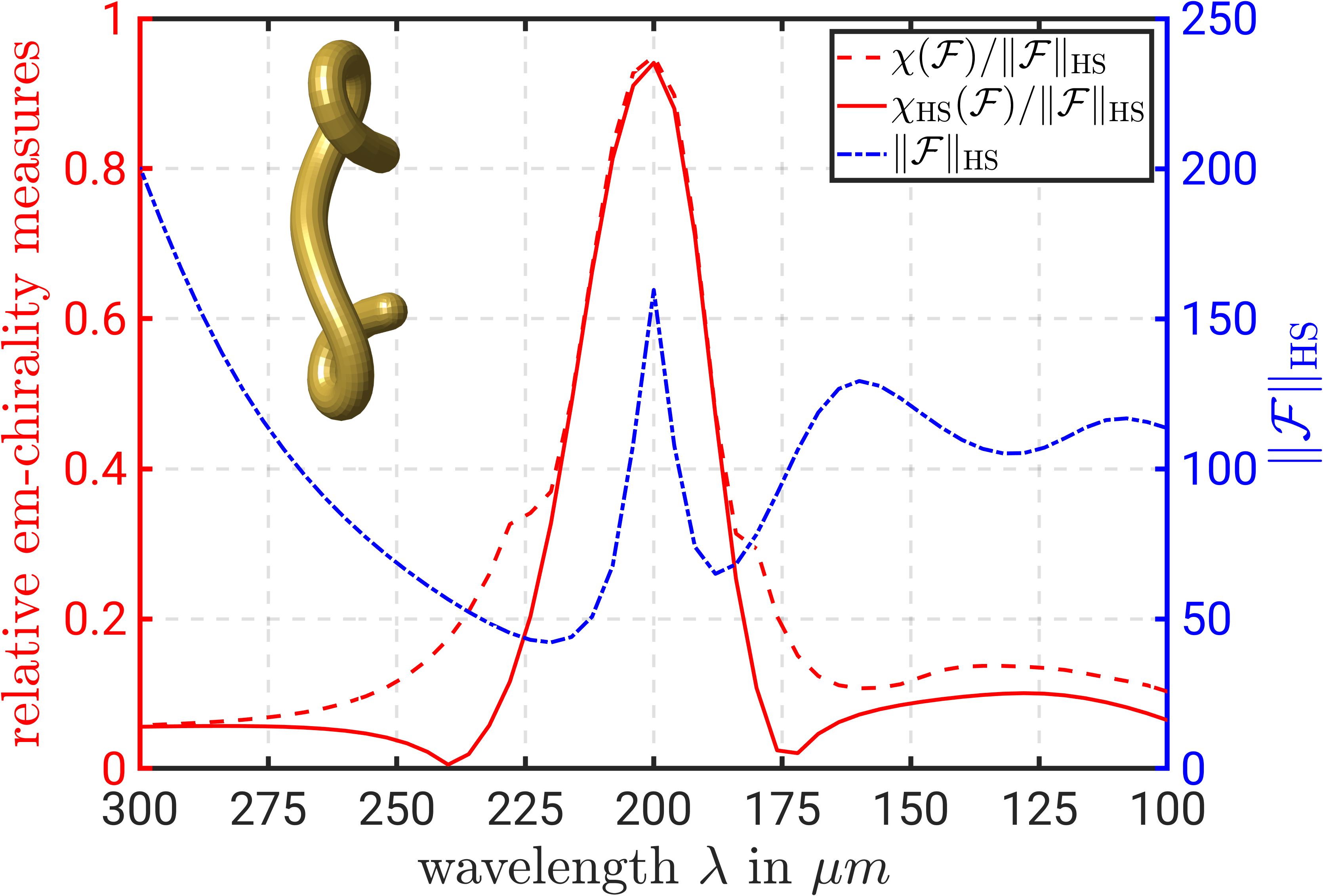} \\
        (c) & (d)
	\end{tabular}
	\caption{
	Frequency scans for wavelengths from $100$-$\SI{300}{\micro\metre}$ of the silver helix from \cite{FerFruRoc2016} (a) and our designs from Example 1--3 ((b) -- (d)). The chirality measure $\chi$ is defined in \cite{FerFruRoc2016}. The helix attains relative em-chirality measure of approximately $92.04\%$ for $\lambda = \SI{200}{\micro\metre}$.}
	\label{fig:freq_scans}
\end{figure}

\section{Conclusions}
\label{sec:Conclusions}

We have described and implemented a shape optimization scheme to generate highly em-chiral perfectly conducting tubular scatterers. The formulation of this algorithm requires the Fr\'echet derivative of the far field operator with respect to the Hilbert-Schmidt norm which we have proved to exist. Compared to our earlier work \cite{AreKnoSch2026} on iteratively solving inverse electromagnetic scattering shape reconstruction problems, we have managed to substantially speed up the computation times for solving the boundary integral equation which was necessary to make our approach feasible.

In our numerical experiments we have been able to robustly generate a number of non-intuitive shapes quite different from standard helices that exhibit very high measures of em-chirality. These obstacles are of a substantially bigger size and hence much easier to produce in practice than previously known examples with comparable physical properties. Our shape optimization approach thus shows how a larger variety of geometries for highly em-chiral objects may be generated while additionally satisfying further necessary design constraints. Attempts to optimize the thickness of the tube along its length have been less successful and will require novel and more sophisticated ideas to formulate suitable regularization terms.

\red{Along the way, we have established a novel uniqueness result for the inverse obstacle scattering problem with incident fields and far field observations of just one helicity.  This result rules out the existence of maximally em-chiral perfectly conducting scattering objects.}

\appendix

\section{Non-existence of perfectly conducting maximally em-chiral objects}

\red{As stated in Theorem \ref{thm:no_max_chiral_object}, there cannot exist any perfectly conducting object that attains the normalized em-chirality measure $\mathcal{J}_\hs = 1$. The argument for this fact is linked to a novel uniqueness result for the inverse electromagnetic obstacle scattering problem with incident fields and far field observations of just one helicity, which we provide in this appendix.
Its proof is a modification of the proof of a corresponding uniqueness theorem without assumptions on the helicity of the incident fields and the far field observations of the scattered fields that has been established in \cite{Kress02}.}

\newcommand{\Pcal}{\mathcal{P}}
\newcommand{\Einfty}{\vE^\infty}
\newcommand{\Hinfty}{\vH^\infty}
\newcommand{\vq}{{\boldsymbol{q}}}

\begin{theorem}
  \label{thm:UniquenessPEC}
  Assume that $D_1$, $D_2 \subseteq \RR^3$ are two perfectly conducting obstacles such that either
  \begin{subequations}
    \label{eq:UniquenessPEC}
    \begin{align}
      \Pcal^+(\Einfty_1(\cdot, \vd, \vp)) & \,=\, \Pcal^+(\Einfty_2(\cdot, \vd, \vp))
      && \text{for all } \vd\in\SS^2 \text{ and } \vp\in\CC^3 \text{ with }
         \I \, (\vd \times \vp) = \vp \,, \label{eq:UniquenessPECa}
    \intertext{or}
      \Pcal^-(\Einfty_1(\cdot, \vd, \vp))
      & \,=\, \Pcal^-(\Einfty_2(\cdot, \vd, \vp))
      &&\text{for all } \vd \in \SS^2 \text{ and } \vp \in \CC^3 \text{ with }
         \I \, (\vd \times \vp) = -\vp \,. \label{eq:UniquenessPECb}
    \end{align}
  \end{subequations}
  Then $D_1=D_2$.
\end{theorem}

\begin{proof}
  From the Silver-M\"uller radiation condition, it follows that the far field pattern of the magnetic field is related to that of the electric field by
  \[
    \vH^\infty(\widehat{\vx}, \vd, \vp) = \widehat{\vx} \times \vE^\infty(\widehat{\vx}, \vd, \vp) \, , \qquad \widehat{\vp}, \vd \in \SS^2 \, , \ \vp \in \CC^3 \, , \ \vd \cdot \vp = 0 \, .
  \]
  Suppose that \eqref{eq:UniquenessPECa} holds. For any $\vd \in \SS^2$ and $\vp \in \CC^3$ such that $\I \, (\vd \times \vp) = \vp$ we obtain using the definition of the projection $\mathcal{P}^+$ in the paragraph below \eqref{eq:V_pm} that
  \begin{equation}
    \label{eq:FarfieldsEqual}
    \Einfty_1(\widehat{\vx}, \vd, \vp) + \I \, \Hinfty_1(\widehat{\vx}, \vd, \vp)
    \,=\, \Einfty_2(\widehat{\vx}, \vd, \vp) + \I \, \Hinfty_2(\widehat{\vx}, \vd, \vp) \,,
    \qquad \widehat{\vx} \in \SS^2 \,.
  \end{equation}
  We denote by
  \begin{equation}
    \label{eq:Dipole}
    \vE^i_e(\vx, \vz, \vq) \,=\, \frac{\I}{k} \Dcurl_{\vx} \Dcurl_{\vx} \vq \Phi(\vx - \vz) \,, \quad
    \vH^i_e(\vx, \vz, \vq) \,=\, \Dcurl_{\vx} \vq\Phi(\vx - \vz) \,,
    \qquad \vx\in\RR^3 \,, \; \vx \not= \vz \,,
  \end{equation}
  an electric dipole with polarization $\vq \in \CC^3$ at $\vz \in \RR^3$. Let $B_R \subseteq \RR^3$ be a sufficiently large ball such that $\overline{D_1} \cup \overline{D_2} \subseteq B_R$.
  Using \cite[Lmm. 3.2]{KirKre93}, we find that for any $z\not\in \overline{B_R}$ the function $\Phi(\cdot-z)$ together with its first and second derivatives can be uniformly approximated on compact subsets of $B_R$ by functions from $\Span\{\E^{\I k \vx \cdot \vd} : \vd\in\SS^2 \}$ (The second derivatives are not treated in \cite[Lmm. 3.2]{KirKre93}, but the result follows in the same way as for the first derivatives).  Accordingly, using \eqref{eq:Dipole} we obtain that
  \begin{equation*}
    \vW_e^i(\cdot, \vz, \vq) \,:=\, \vE^i_e(\cdot, \vz, \vq) + \I \vH^i_e(\cdot, \vz, \vq) \,, \qquad
    \vq \in \CC^3 \,, \quad \vz\not\in \overline{B_R} \,,
  \end{equation*}
  can be uniformly approximated on compact subsets of $B_R$ by functions from
  \begin{equation*}
    \Span \bigl\{ \vE^i(\cdot, \vd, \vp) + \I \vH^i(\cdot, \vd, \vp) \;\colon\;
    \vd \in\SS^2 \,,\; \vp \in \CC^3 \setminus \{ 0 \} \, , \; \vd \cdot \vp = 0 \bigr\} \,.
  \end{equation*}
  For any $\vd \in \SS^2$ and $\vp \in \CC^3 \setminus \{ 0 \}$, $\vd \cdot \vp = 0$, we can decompose $\vp = \vp_+ + \vp_-$ with $\vp_\pm := \vp \pm \I \, (\vd \times \vp)$ and we have that $\vE^i(\cdot, \vd, \vp) + \I \vH^i(\cdot, \vd, \vp) = \vE^i(\cdot, \vd, \vp_+)+\I \vH^i(\cdot, \vd, \vp_+)$. Therefore, $W_e^i(\cdot, \vz, \vp)$ with $\vz \not\in \overline{B_R}$ can in fact be uniformly approximated on compact subsets of $B_R$ by functions from
  \begin{equation*}
    \Span\bigl\{ \vE^i(\cdot, \vd, \vp) + \I \, \vH^i(\cdot, \vd, \vp) \;:\;
    \vd \in \SS^2 \,, \quad \vp \in \CC^3 \setminus \{ 0 \} \,, \quad \I \, (\vd \times \vp) = \vp \bigr\} \,.
  \end{equation*}

  Next, we denote by $\vW^s_{e,j}(\cdot, \vz, \vq)$, $j=1,2$, the unique radiating  solution to the exterior boundary value problem
  \begin{equation}
    \label{eq:MaxwellWsej}
    \Dcurl \Dcurl \vW^s_{e,j} - k^2 \vW^s_{e,j} \,=\, 0
    \quad \text{in } \RR^3 \setminus \overline{D_j} \,, \qquad
    \vnu\times \vW^s_{e,j} \,=\, -\vnu\times \vW^i_e \quad \text{on } \partial D_j \,,
  \end{equation}
  respectively. Using the continuous dependence of the far field pattern $\vW^\infty_{e,j}(\cdot, \vz, \vq)$ of $\vW^s_{e,j}(\cdot, \vz, \vq)$ on the boundary data $\nu\times \vW^s_{e,j}(\cdot, \vz, \vq)|_{\partial D_j}$, $j=1,2$, we conclude   from \eqref{eq:FarfieldsEqual} and the density result which we just have established that
  \begin{equation}
    \label{eq:FarfieldsWEqual}
    \vW^\infty_{e,1}(\widehat{\vx}, \vz, \vq)
    \,=\, \vW^\infty_{e,2}(\widehat{\vx}, \vz, \vq) \, , \qquad
    \widehat{\vx} \in \SS^2 \,, \quad \vz \not \in \overline{B_R} \,, \quad \vq \in \CC^3 \, .
  \end{equation}
  Let $\Omega_c \subseteq \RR^3$ denote the unbounded component of the complement of $\overline{D_1} \cup \overline{D_2}$. Since the far field patterns $\vW^\infty_{e,j}(\cdot, \vz, \vq)$, $j=1,2$, depend analytically on the parameter $\vz$ for $\vz \in \Omega_c$, the identity \eqref{eq:FarfieldsWEqual} in fact holds for all $\vz \in \Omega_c$. Rellich's lemma \cite[Thm.~6.10]{ColKre2019} shows that
  \begin{equation*}
    \vW^s_{e,1}(\vx, \vz, \vq)
    \,=\, \vW^s_{e,2}(\vx, \vz, \vq) \,, \qquad \vx, \, \vz \in \Omega_c \,, \quad \vq \in \CC^3 \,.
  \end{equation*}

  Now assume that $D_1\not= D_2$. Without loss of generality, there exists $\vx_*\in \partial \Omega_c$ such that
  $\vx_* \in \partial D_1$ and $x_* \not\in \overline{D_2}$.
  In particular we have
  \begin{equation*}
    \vz_n \,=\, \vx_* + \frac{1}{n} \vnu(x_*) \in \Omega_c \,, \qquad n=1,2,\ldots,
  \end{equation*}
  for sufficiently large $n$, where $\vnu(x_*)$ denotes the unit outward normal vector on $\partial  D_1$ in $vx_*$. Then, on the one hand we obtain that, for any $\vq \in \CC^3$,
  \begin{equation*}
    \lim_{n\to\infty} \vnu(\vx_*) \times \vW^s_{e,2}(\vx_*, \vz_n, \vq) \big|_{\partial D_1}
    \,=\, \vnu(\vx_*) \times \vW^s_{e,2}(\vx_*, \vx_*, \vq) \big|_{\partial D_1}
  \end{equation*}
  since $\vW^s_{e,2}(\vx_*, \cdot, \vq)$ is continuous in a neighborhood of $\vx_* \not\in \overline{D_2}$.
  On the other hand, we find that
  \begin{equation*}
    \lim_{n\to\infty} \bigl| \vnu(\vx_*) \times \vW^s_{e,1}(\vx_*, \vz_n, \vq) \big|_{\partial D_1} \bigr|
    \,=\, \lim_{n\to\infty}\bigl| \vnu(x_*) \times \vW^i_e(\vx_*, \vz_n, \vq) \big|_{\partial D_1} \bigr|
    \,=\, \infty
  \end{equation*}
  for $\vq \perp \vnu(\vx_*)$ with $\vq\not=0$, because of the boundary condition $\vnu \times \vW^s_{e,1}(\cdot, \vz_n, \vq) = - \vnu \times \vW^i_e(\cdot, \vz_n, \vq)$ on   $\partial D_1$ and since a short calculation exploiting the singularity in the fundamental solution shows that
  \begin{equation*}
    \bigl| \vnu(x_*) \times \vW^i_e(\vx_*, \vz_n, \vq) \big|_{\partial D_1} \bigr|
    \,=\,  \frac{|\vq|}{4\pi k^3} \, n^3 + \mathrm{O}(n^2)
    \qquad \text{as } n \to \infty \,.
  \end{equation*}
  This contradicts that
  \begin{equation*}
    \vnu(\vx_*) \times \vW^s_{e,1}(\vx_*, \vz_n, \vq) \big|_{\partial D_1}
    \,=\, \vnu(\vx_*) \times \vW^s_{e,2}(\vx_*, \vz_n, \vq) \big|_{\partial D_1}
  \end{equation*}
  for all $n\in\NN$. Thus $D_1=D_2$.

  The second case when \eqref{eq:UniquenessPECb} holds can be treated analogously.
\end{proof}

\section*{Acknowledgments}
Funded by the Deutsche Forschungsgemeinschaft (DFG, German Research
Foundation) -- Project-ID 258734477 -- SFB 1173.

Marvin Kn\"oller was supported by the Research Council of Finland (Flagship of Advanced Mathematics
for Sensing, Imaging and Modelling grant 359182).

\small
\bibliographystyle{abbrvurl}
\bibliography{opt-design-tubular-pc-objects}
\end{document}